\documentclass[a4paper, 12pt]{amsart}

\usepackage{mathptmx, amsmath, tabularx, amsthm, amssymb, amsfonts, amscd}
\usepackage{eulervm}
\usepackage{paralist}
\usepackage{aliascnt}
\usepackage{blkarray}
\usepackage{mathbbol}
\usepackage[inner=2.3cm,outer=2.3cm, bottom=3.3cm]{geometry}
\usepackage{setspace}
\usepackage[colorlinks=true,linkcolor=blue,urlcolor=blue,pagebackref]{hyperref}
\usepackage[initials]{amsrefs}
\usepackage{marginnote}
\usepackage{tikz}
\usetikzlibrary{matrix}
\usetikzlibrary{arrows,calc}
\allowdisplaybreaks

\BibSpec{collection.article}{%
	+{}  {\PrintAuthors}                {author}
	+{,} { \textit}                     {title}
	+{.} { }                            {part}
	+{:} { \textit}                     {subtitle}
	+{,} { \PrintContributions}         {contribution}
	+{,} { \PrintConference}            {conference}
	+{}  {\PrintBook}                   {book}
	+{,} { }                            {booktitle}
	+{,} { }                            {series}
	+{, vol.} { }                            {volume}
	+{,} { }                            {publisher}
	+{,} { \PrintDateB}                 {date}
	+{,} { pp.~}                        {pages}
	+{,} { }                            {status}
	+{,} { \PrintDOI}                   {doi}
	+{,} { available at \eprint}        {eprint}
	+{}  { \parenthesize}               {language}
	+{}  { \PrintTranslation}           {translation}
	+{;} { \PrintReprint}               {reprint}
	+{.} { }                            {note}
	+{.} {}                             {transition}
	+{}  {\SentenceSpace \PrintReviews} {review}
}

\newcommand{\RR}{\mathbb{R}}
\newcommand{\kk}{\mathbb{k}}

\newcommand{\NN}{\normalfont\mathbb{N}}
\newcommand{\ZZ}{{\normalfont\mathbb{Z}}}

\newcommand{\dd}{{\normalfont\mathbf{d}}}

\newcommand{\mm}{{\normalfont\mathfrak{m}}}

\newcommand{\QQ}{\mathbb{Q}}

\newcommand{\bn}{{\normalfont\mathbf{n}}}

\newcommand{\bx}{{\normalfont\mathbf{x}}}

\newcommand{\bm}{{\normalfont\mathbf{m}}}

\newcommand{\vol}{\normalfont\text{vol}}

\newcommand{\II}{\mathbb{I}}

\newcommand{\Rees}{\mathcal{R}}

\newcommand{\ee}{{\normalfont\mathbf{e}}}

\newcommand{\JJ}{\mathbb{J}}

\newcommand{\BB}{\mathbb{B}}

\newcommand{\bJ}{\mathbf{J}}
\newcommand{\bK}{\mathbf{K}}

\newcommand{\FF}{\normalfont\mathcal{F}}

\newcommand{\Proj}{\normalfont\text{Proj}}

\newcommand{\Vol}{{\normalfont\text{Vol}}}
\newcommand{\MV}{{\normalfont\text{MV}}}
\newcommand{\ind}{{\normalfont\text{ind}}}
\newcommand{\Con}{{\normalfont\text{Con}}}

\newcommand{\multProj}{\normalfont\text{MultiProj}}

\def\f0{\mathbf{0}}

\def\fn{\mathbf{n}}
\def\fm{\mathbf{m}}

\def\fd{\mathbf{d}}

\def\ls{\leqslant}
\def\gs{\geqslant}

\def\*{{\color{red}\blacksquare}}

\newtheorem{theorem}{Theorem}[section]

\newtheorem{headthm}{Theorem}

\newaliascnt{headcor}{headthm}

\aliascntresetthe{headcor}

\newaliascnt{headthmdef}{headthm}

\aliascntresetthe{headthmdef}

\newaliascnt{headconj}{headthm}

\aliascntresetthe{headconj}

\newaliascnt{corollary}{theorem}

\aliascntresetthe{corollary}

\newaliascnt{lemma}{theorem}
\newtheorem{lemma}[lemma]{Lemma}
\aliascntresetthe{lemma}

\newaliascnt{conjecture}{theorem}

\aliascntresetthe{conjecture}

\newaliascnt{proposition}{theorem}
\newtheorem{proposition}[proposition]{Proposition}
\aliascntresetthe{proposition}

\theoremstyle{definition}
\newaliascnt{definition}{theorem}
\newtheorem{definition}[definition]{Definition}
\aliascntresetthe{definition}

\newaliascnt{notation}{theorem}
\newtheorem{notation}[notation]{Notation}
\aliascntresetthe{notation}

\newaliascnt{example}{theorem}

\aliascntresetthe{example}

\newaliascnt{examples}{theorem}

\aliascntresetthe{examples}

\newaliascnt{remark}{theorem}
\newtheorem{remark}[remark]{Remark}
\aliascntresetthe{remark}

\newaliascnt{problem}{theorem}

\aliascntresetthe{problem}

\newaliascnt{question}{theorem}

\aliascntresetthe{question}

\newaliascnt{convention}{theorem}

\aliascntresetthe{convention}

\newaliascnt{construction}{theorem}

\aliascntresetthe{construction}

\newaliascnt{setup}{theorem}
\newtheorem{setup}[setup]{Setup}
\aliascntresetthe{setup}

\newaliascnt{algorithm}{theorem}

\aliascntresetthe{algorithm}

\newaliascnt{observation}{theorem}

\aliascntresetthe{observation}

\newaliascnt{defprop}{theorem}

\aliascntresetthe{defprop}

\def\equationautorefname~#1\null{(#1)\null}
\def\sectionautorefname~#1\null{Section #1\null}
\def\subsectionautorefname~#1\null{\S #1\null}

\begin{document}

\title[Convex bodies and graded families of monomial ideals]{Convex bodies and  graded families of monomial ideals}

\author[Yairon Cid-Ruiz]{Yairon Cid-Ruiz}
\address[Cid-Ruiz]{Department of Mathematics: Algebra and Geometry, Ghent University, Krijgslaan 281 – S25, 9000 Ghent, Belgium}
\email{Yairon.CidRuiz@ugent.be}
\urladdr{https://ycid.github.io}

\author[Jonathan Monta\~no]{Jonathan Monta\~no$^{*}$}
\address[Monta\~no ]{Department of Mathematical Sciences  \\ New Mexico State University  \\PO Box 30001\\Las Cruces, NM 88003-8001}
\thanks{$^{*}$ The second  author is  supported by NSF Grant DMS \#2001645.}
\email{jmon@nmsu.edu}

\date{\today}

\begin{abstract}
We show that the mixed volumes of arbitrary convex bodies are equal to  mixed multiplicities of  graded families of monomial ideals, and  to  normalized limits of mixed multiplicities of monomial ideals. This result evinces the close relation between  the theories of  mixed volumes from convex geometry and mixed multiplicities from commutative algebra.

\end{abstract}

\maketitle

\section{Introduction}

The connection between   volumes of convex bodies and algebraic-geometric invariants  has long been explored by researchers and it has led to numerous applications in various fields of mathematics.   
To highlight some of these we have  Bernstein–Koushnirenko–Khovanskii Theorem  \cite{bernshtein, kho78, kouchnirenko}, Hu's proof of the log-concavity of  characteristic polynomials of matroids \cite{Huh12}, and the theory of Newton-Okounkov bodies \cite{KAVEH_KHOVANSKII, lazarsfeld09} and its applications to limits in commutative algebra \cite{cutkosky2014}. 

\medskip

 The goal of this paper is to expand  on this fruitful research line by showing  that the mixed volumes of arbitrary convex bodies are equal to  mixed multiplicities of  graded families of monomial ideals, and also to normalized limits of  mixed multiplicities of   monomial ideals  (\autoref{thmC}). This  is an extension of the main result in  \cite{TRUNG_VERMA_MIXED_VOL} where the case of lattice polytopes is treated.  Our proof is  based on two intermediate results of interest in their own right (\autoref{thmA} and \autoref{thmB}). 

\medskip

Let $R$ be a  $d$-dimensional standard graded polynomial ring over a field $\kk$  and  $\mm=[R]_+$ its graded irrelevant ideal. 
For  homogeneous ideals $J_1,\ldots, J_r$ and an $\mm$-primary homogeneous ideal $I$ there exist integers $e_{(d_0,\fd)}(I|J_1, \ldots, J_r)\gs 0$  for every $d_0\in \NN$, $\fd=(d_1,\ldots, d_r)\in \NN^r$ with $d_0+|\fd|=d-1$, called the {\it mixed multiplicities} of $J_1,\ldots, J_r$ with respect to  $I$, such that  
$$\lim_{m\rightarrow \infty}\frac{\dim_\kk \left( J_1^{mn_1}\cdots J_r^{mn_r}/I^{n_0m}J_1^{mn_1}\cdots J_r^{mn_r} \right)}{m^{d}}=\sum_{d_0+|\fd|=d-1}\frac{e_{(d_0,\fd)}(I|J_1,\ldots, J_r )}{(d_0+1)!d_1!\cdots d_r!}n_0^{d_0+1}n_1^{d_1}\cdots n_r^{d_r},$$
for every $n_0,n_1,\ldots, n_r\gs 0$ (see \cite{TRUNG_VERMA_SURVEY} for a survey). 
 On the other hand,  by Minkowski's Theorem,  for  a sequence of convex bodies $K_1,\ldots,K_r$   in $\RR^d$,  the {\it mixed volumes} $\MV_d(K_{\rho_1},\ldots,K_{\rho_d})$ of sequences  $(K_{\rho_1},\ldots,K_{\rho_d})$ of convex bodies with $1\ls \rho_1,\ldots,\rho_d\ls r$ satisfy the following equation
$$\Vol_{d}(\lambda_1 K_1+\cdots+\lambda_r K_r) = \sum_{\substack{\dd =(d_1,\ldots,d_r) \in \NN^{r}\\ d_1+\cdots+d_r=d}} \frac{1}{d_1!\cdots d_r!} \MV_d\left(\bigcup_{i=0}^r\bigcup_{j=1}^{d_i} \{K_i\}\right)\lambda_{1}^{d_1}\cdots \lambda_{r}^{d_r}$$
for every $\lambda_1,\ldots, \lambda_r\gs 0$  (see \cite[Theorem 3.3, page 116]{EWALD}). The relation between these two sets of invariants is  established in \cite{TRUNG_VERMA_MIXED_VOL} where the authors show that mixed volumes of lattice polytopes coincide with the mixed multiplicities of certain monomial ideals. 
However, this relation does not extend to arbitrary convex bodies as the associated sequences of ideals are no longer powers of ideals but rather  (not necessarily Noetherian) graded families  of ideals.

\medskip

A sequence of ideals $\II=\{I_n\}_{n\in \NN}$ is a {\it graded family} if $I_0=R$ and $I_iI_j\subseteq I_{i+j}$ for every $i,j\in \NN$. 
The family is {\it Noetherian} if the corresponding Rees algebra $\Rees(\II) =\oplus_{n\in \NN}I_nt^n\subseteq R[t]$ is Noetherian. The study of mixed multiplicities of graded families  was pioneered by Cutkosky-Sarkar-Srinivasan \cite{cutkosky2019} for the case of \emph{$\mm$-primary filtrations} (in  more general rings), that is, when each $I_n$ is $\mm$-primary and $I_{n+1}\subseteq I_n$ for every $ n \in \NN$.  
Their strategy is to first show the existence of these multiplicities for Noetherian filtrations, and then pass to  arbitrary filtrations using the theory of Newton-Okounkov bodies \cite{KAVEH_KHOVANSKII} (see also  \cite{cutkosky2014} and \cite{lazarsfeld09}).  
In our first result, we prove the existence of mixed multiplicities for arbitrary graded families of monomial ideals under a mild assumption. Our approach differs from the one of \cite{cutkosky2019} in that we exploit Minkowski's Theorem to show the existence of the  polynomial leading to the definition of mixed-multiplicities. 

\medskip

In order to present  our first result  we need to introduce some prior notation.	Let $\II = \{I_n\}_{n \in \NN}$ be a  (not necessarily Noetherian) graded family of $\mm$-primary monomial ideals and  let $\JJ(1)=\{J(1)_n\}_{n\in \NN}, \ldots, \JJ(r)=\{J(r)_n\}_{n\in \NN}$ be (not necessarily Noetherian) graded families of monomial ideals in $R$. 
We further assume that the degrees of generators of $J(i)_n$ are bounded by a linear function on $n$ for each $1\ls i\ls r$. 
We note that the latter condition is similar to others that have been considered in previous works regarding limits of graded families of ideals (see, e.g., \cite[Theorem 6.1]{cutkosky2014}).

\begin{headthm}[{\autoref{thm_general_case}, \autoref{lem_non_neg_general}}]\label{thmA}
	Under the notations and assumptions above, the function 
	$$
	F(n_0,n_1,\ldots,n_r) = \lim_{m\to \infty} \frac{\dim_\kk\big(J(1)_{mn_1}\cdots J(r)_{mn_r}\,\big/\,I_{mn_0}J(1)_{mn_1}\cdots J(r)_{mn_r}\big)}{m^d}
	$$
	is equal to a homogeneous polynomial $G(n_0,\bn) = G(n_0,n_1,\ldots,n_r)$ of total degree $d$ with non-negative real coefficients for all $n_0 \in \NN$ and $\bn = (n_1,\ldots,n_r) \in \NN^r$. 
	Additionally, $G(n_0,\bn)$ has no term of the form $\alpha \bn^\dd = \alpha n_1^{d_1}\cdots n_r^{d_r}$ with $0 \neq \alpha \in \RR$, $\dd =(d_1,\ldots,d_r) \in \NN^r$ and $|\dd|=d$.
	
	Furthermore,  the coefficients of  the polynomial $G(n_0,\bn)$ can be explicitly described in terms of mixed volumes of  certain Newton-Okounkov bodies.
\end{headthm}

We note that \autoref{thmA} is new even when the graded families  are all $\mm$-primary. In this case,  our theorem is an extension of that of \cite{cutkosky2019} for monomial ideals (also, see \cite{kaveh2014}). 
For this reason, we isolate the $\mm$-primary case in \autoref{thm_m_primary_case}.  

\medskip

The polynomial  $G(n_0, \bn)$ from \autoref{thmA}  can be written as 
	$$
	G(n_0,\bn) = \sum\limits_{\substack{(d_0,\dd)\in \NN^{r+1} \\ d_0+|\dd| = d-1}} \frac{1}{(d_0+1)!\dd!}\, e_{(d_0,\dd)}\left(\II|\JJ(1),\ldots,\JJ(r)\right)\, n_0^{d_0+1}\bn^\dd,
	$$
here if $\dd = (d_1,\ldots,d_r)$ then  $\dd!=d_!!\cdots d_r!$. For each $(d_0,\dd)\in \NN^{r+1}$ with $d_0+|\dd| = d-1$, we define the real numbers $e_{(d_0,\dd)}\left(\II|\JJ(1),\ldots,\JJ(r)\right)\ge 0$ to be the {\it mixed multiplicities} of    $\JJ(1), \ldots, \JJ(r) $ with respect to $\II $ (see \autoref{def_mix_mult_general}).

\medskip

The {\it volume} and {\it multiplicity} of a graded family  of zero dimensional ideals $\BB=\{B_n\}_{n\in \NN}$ in a Noetherian local ring $S$ of dimension $s$ are defined as 
$$\vol_S(\BB)=\limsup_{n\to \infty}\frac{\lambda(S/B_n)}{n^{s}/s!} \quad \text{and}\quad  e_S(\BB)=\lim_{p\to \infty}\frac{e_S(B_p)}{p^{s}},$$ 
respectively; 
where $\lambda(N)$ denotes length of an $S$-module $N$ and $e_S(J)$ denotes the Hilbert-Samuel multiplicity of an ideal $J$. 
Several works prove the equality of these two invariants under  certain assumptions (see \cite{ein2003,must02,lazarsfeld09,cutkosky2013,cutkosky2014}). The general version of the so called {\it Volume = Multiplicity formula} is due to Cutkosky and it is shown on  any $S$ for which the limit in the definition of volume exists \cite{cutkosky2015}. In our next result we show a ``Volume =  Multiplicity formula'' for mixed multiplicities of  graded families of monomial ideals.

\begin{headthm}[\autoref{thm_mult_eq_vol}]\label{thmB}
With the above assumptions and notations, for each $d_0 \in \NN$ and $\dd = (d_1,\ldots,d_r) \in \NN^r$ with $d_0+|\dd| = d-1$, we have the equality 
	$$
	e_{(d_0,\dd)}\left(\II|\JJ(1),\ldots,\JJ(r)\right) \;=\; \lim_{p\to \infty} \frac{e_{(d_0,\dd)}\left(I_p|J(1)_p,\ldots,J(r)_p\right)}{p^d}.
	$$
\end{headthm}

With the previous results in hand, we are ready to present  the main result of this paper. 
Here we express mixed volumes of arbitrary convex bodies as  mixed multiplicities of graded families of monomial ideals and as normalized  limits of mixed-multiplicities of ideals. 

\medskip

Now, let us fix the following slightly different notation.  
Let $(K_1,\ldots,K_r)$ be a sequence of convex bodies in $\RR_{\ge 0}^d$ and $K_0 \subset \RR^{d}$ be the convex hull of the points $\mathbf{0}, \ee_1,\ldots,\ee_{d} \in \RR^{d}$, where $\mathbf{0} = (0,\ldots,0) \in \RR^d$ and $\ee_i = (0,\ldots,1,\ldots,0)$ denotes the $i$-th elementary basis vector for $1 \le i \le d$. 
Denote by $\bK$ the sequence of convex bodies $\bK = (K_0,K_1,\ldots,K_r)$. For each $(d_0,\dd) = (d_0,\ldots, d_r)\in \NN^{r+1}$ we let $\bK_{(d_0,\dd)}$ be the multiset $\bK_{(d_0,\dd)} = \bigcup_{i=0}^r\bigcup_{j=1}^{d_i} \{K_i\}$ of $d_i$ copies of $K_i$ for each $0 \le i \le r$. 
Here let $R$ be a $(d+1)$-dimensional standard graded polynomial ring over a field $\kk$ and $\mm = [R]_+$ be its graded irrelevant ideal.  We let $\mathbb{M}$ be the graded family $\mathbb{M}=\{\mm^n\}_{n\in \NN}$.

\begin{headthm}[\autoref{mainCor}]\label{thmC}
Under the notations and assumptions above, there exist graded families of monomial ideals $\JJ(1)\ldots, \JJ(r)$ in $R$ such that, for each $(d_0,\dd) \in \NN^{r+1}$ with $d_0 + \lvert \dd \rvert = d$, we have the equalities 
\begin{align*}
	\MV_d\left(\bK_{(d_0,\dd)}\right) &=  e_{(d_0,\dd)}(\mathbb{M} \mid \JJ(1),\ldots,\JJ(r)) \\
	&= \lim_{p\to \infty} \frac{e_{(d_0,\dd)}(\mm^p \mid J(1)_p,\ldots,J(r)_p)}{p^{d+1}} \;=\; \lim_{p\to \infty} \frac{e_{(d_0,\dd)}(\mm \mid J(1)_p,\ldots,J(r)_p)}{p^{|\dd|}}.
\end{align*}
In particular, when $r = d$, we obtain the equalities
\begin{align*}
	\MV_d(K_1,\ldots,K_d) &=  e_{(0,1,\ldots,1)}(\mathbb{M} \mid \JJ(1),\ldots,\JJ(d)) \\
	&= \lim_{p\to \infty} \frac{e_{(0,1,\ldots,1)}(\mm^p \mid J(1)_p,\ldots,J(d)_p)}{p^{d+1}} \;=\; \lim_{p\to \infty} \frac{e_{(0,1,\ldots,1)}(\mm \mid J(1)_p,\ldots,J(d)_p)}{p^{d}}.
\end{align*}
\end{headthm}

Finally, we briefly describe the content of the paper. 
In \autoref{SectionPrelim} we set up the notation and include some preliminary results that are used in the rest of the paper. In  \autoref{sec_Mixed_Mult} we include the proof of \autoref{thmA} and in \autoref{subsect_Vol_eq_Mult} the one of  \autoref{thmB}. 
Lastly, \autoref{sec_mixed_vol_mixed_mult}  includes the proof of our main result \autoref{thmC}.

\section{Notation and Preliminaries}\label{SectionPrelim}

In this section, we set up the notation that is used throughout the article. We also include some preliminary information needed for our results.

For a vector $\fn=(n_1,\ldots, n_r)\in \NN^r$ we denote by $|\fn|$ the sum of its entries. We also denote by $\mathbf{0}$ and $\mathbf{1}$ the vectors $(0,\ldots, 0)\in \NN^r$ and $(1,\ldots, 1)\in \NN^r$, respectively. 
For $\fn=(n_1\ldots, n_r)$ and $\fm=(m_1,\ldots, m_r)$ in $\NN^r$ we write $\fn\gs \fm$ if $n_i\gs m_i$ for every $1\ls i\ls r$. Moreover, we write $\fn\gg \mathbf{0}$ if $n_i\gg 0$ for every $1\ls i\ls r$. We also use the abbreviations $\fn^\fm=n_1^{m_1}\cdots n_r^{m_r}$ and $\fn!=n_1!\cdots n_r!$.

Below we recall the definitions of graded families of ideals and filtrations of ideals.

\begin{definition}
	A \emph{graded family} of ideals $\{I_i\}_{i \in \NN}$ in a ring $R$ is a family of ideals indexed by the natural numbers such that $I_0 = R$ and $I_iI_j  \subset I_{i+j}$ for all $i, j \in \NN$. 
	\begin{enumerate}[(i)]
		\item If $(R,\mm)$ is a local ring (or a positively graded ring with $\mm = [R]_+$) and $I_i$ is $\mm$-primary for $i > 0$, then we  say that $\{I_i\}_{i \in \NN}$ is a \emph{graded family of $\mm$-primary ideals}.
		
		\item If we have the inclusion $I_i \supseteq I_{i+1}$ for all $i \in \NN$, then we  say that $\{I_i\}_{i \in \NN}$ is a \emph{filtration of ideals} in $R$.
		
		\item We say that $\{I_i\}_{i \in \NN}$ is \emph{Noetherian}  when the corresponding Rees algebra $\bigoplus_{i \in \NN} I_it^i \subset R[t]$ is Noetherian.
		
		\item When $R=\kk[x_1,\ldots,x_d]$ is a standard graded polynomial ring over a field $\kk$ and each $I_i$ is a monomial ideal, we say that $\{I_i\}_{i \in \NN}$ is a \emph{graded family of monomial ideals}.
 	\end{enumerate}	
\end{definition}

\subsection{Mixed volumes of convex bodies}\label{sub_mixed_vol}

Let $\bK = (K_1,\ldots,K_r)$ be a sequence of convex bodies in $\RR^d$.
For any sequence $\lambda = (\lambda_1,\ldots,\lambda_r) \in \NN^r$ of non-negative integers, we denote by $\lambda \bK$ the Minkowski sum $\lambda \bK := \lambda_1 K_1+\cdots+\lambda_r K_r$ and by $\bK_\lambda$ the multiset $\bK_\lambda := \bigcup_{i=1}^r\bigcup_{j=1}^{\lambda_i} \{K_i\}$ of $\lambda_i$ copies of $K_i$ for each $1 \le i \le r$.

For any convex body $K \subset \RR^d$, we denote by $\Vol_{d}(K)$ the $d$-dimensional volume.
The following important and classical theorem says that the volume $\Vol_{d}(\lambda \bK)$ of the convex body $\lambda \bK$ is a polynomial of degree $d$ in $\lambda$ (see \cite[Theorem 3.2, page 116]{EWALD}).
For more details regarding the topic of mixed volumes the reader is referred to \cite[Chapter IV]{EWALD}.

\begin{theorem}[Minkowski]\label{Minkowski}
	\label{thm_Minkowski_formula}
	$\Vol_{d}(\lambda \bK)$ is a homogeneous polynomial of degree $d$ that satisfies 	
	$$
	\Vol_{d}(\lambda_1 K_1+\cdots+\lambda_r K_r) = \sum_{\rho_1=1,\ldots,\rho_d=1}^r V(K_{\rho_1},\ldots,K_{\rho_d})\lambda_{\rho_1},\ldots,\lambda_{\rho_d},
	$$
for certain coefficients $V(K_{\rho_1},\ldots,K_{\rho_d})$,	where the summation is carried out independently over the $\rho_i$ for $1 \le i \le d$.
\end{theorem}

\autoref{Minkowski}  leads to the following definition (see \cite[Theorem 3.3, page 116]{EWALD}).
\begin{definition}
	The \emph{mixed volume} of $d$ convex bodies $K_1,\ldots,K_d \subset \RR^d$ is defined by 
	$$
	\MV_d(K_1,\ldots,K_d) \,:=\, d!V(K_1,\ldots,K_d).
	$$
\end{definition}

Note that under the current notations we have the following equation
\begin{equation}
	\label{eq_volume_poly}
	\Vol_{d}(\lambda \bK) = \sum_{\substack{\dd \in \NN^{r}\\\lvert \dd \rvert=d}}\; \frac{1}{\dd!}\,\MV_d(\bK_\dd) \, \lambda^\dd.
\end{equation}

\subsection{Semigroups, Newton-Okounkov bodies, and Limits of Lengths}\label{sub_NOB}

In this subsection, we describe the notions and methods of Newton-Okounkov bodies and recall some important results from \cite{KAVEH_KHOVANSKII}.

\smallskip

Here we  use a slightly simpler setting.
Suppose that $S \subset \ZZ^{d+1}$ is a semigroup in $\ZZ^{d+1}$.
Fix a linear map $\pi  : \RR^{d+1} \rightarrow \RR$ with integral coefficients, that is $\pi(\ZZ^{d+1}) \subset \ZZ$.

Let $L = L(S)$ be the linear subspace of $R^{d+1}$ which is generated by $S$.
Let $M = M(S)$ be the rational half-space $M(S):= L(S) \cap \pi^{-1}(\RR_{\ge 0})$, and let $\partial M_\ZZ = \partial M \cap \ZZ^{d+1}$.
Let $\Con(S) \subset L(S)$ be the closed convex cone which is the closure of the set of all linear combinations $\sum_i \lambda_is_i$ with $s_i \in S$ and $\lambda_i \ge 0$.
Let $G(S) \subset L(S)$ be the group generated by $S$.

We say that the pair $(S,M)$ is \emph{admissible} if $S \subset M$;
additionally, if $\Con(S)$ is strictly convex and intersects the space $\partial M$ only at the origin, then $(S,M)$ is called a \emph{strongly admissible} pair (see \cite[Definition 1.9]{KAVEH_KHOVANSKII}).

Following \cite{KAVEH_KHOVANSKII}, when $(S,M)$ is an admissible pair we fix the following notation:
\begin{itemize}[--]
	\item $[S]_k := S \cap \pi^{-1}(k)$.
	\item $m = \ind(S,M) := \left[\ZZ : \pi(G(S))\right]$.
	\item $\ind(S,\partial M) := \left[\partial M_\ZZ : G(S) \cap \partial M\right]$.
	\item $\Delta(S, M) := \Con(S) \cap \pi^{-1}(m)$ \quad  (the Newton-Okounkov body of $(S,M)$).
	\item $q = \dim(\partial M)$.
	\item $\Vol_q(\Delta(S,M))$ is the \emph{integral volume} of $\Delta(M,S)$ (see \cite[Definition 1.13]{KAVEH_KHOVANSKII}); this volume is computed using the
	translation of the \emph{integral measure} on $\partial M$.
\end{itemize}

The following result is of fundamental importance in our approach.

\begin{theorem}[{Kaveh-Khovanskii, \cite[Corollary 1.16]{KAVEH_KHOVANSKII}}]\label{thm_limit_KK}
	Suppose that $(S,M)$ is strongly admissible. 
	Then 
	$$
	\lim_{k\to \infty} \frac{\#[S]_{km}}{k^q} \;=\; \frac{\Vol_q(\Delta(S,M))}{\ind(S,\partial M)}.
	$$
\end{theorem}

\begin{remark}
	Whenever the rational half-space $M$ is implicit from the context, we  write $\Delta(S)$ instead of $\Delta(S,M)$.
\end{remark}

%
%
%

\section{Mixed Multiplicities of graded families of monomial ideals}\label{sec_Mixed_Mult}

	Throughout the present section  we use the data below.
	
	\begin{setup}
		\label{setup_mixed_mult}
		Let $\kk$ be a field, $R$ be the standard graded polynomial ring $R = \kk[x_1,\ldots,x_d]$, and $\mm \subset R$ be the graded irrelevant ideal $\mm=\left(x_1,\ldots,x_d\right)$.
		 Following the notation in \autoref{sub_NOB}, we fix the linear map $\pi : \RR^{d+1} \rightarrow \RR$ given by the projection $(\alpha_1,\ldots,\alpha_d,\alpha_{d+1}) \in \RR^{d+1} \mapsto \alpha_{d+1} \in R$.
		Let $\pi_1 : \RR^{d+1} \rightarrow \RR^{d}$ be the projection given by $(\alpha_1,\ldots,\alpha_d,\alpha_{d+1}) \in \RR^{d+1} \mapsto (\alpha_1,\ldots,\alpha_d) \in \RR^d$.
		Let $M$ be the rational half-space $M = \pi^{-1}(\RR_{\gs 0})=\RR^d\times \RR_{\gs 0}$.			
		
	For a semigroup $S\subset \NN^{d+1}$ and $m \in \NN$, we denote by $[S]_{m}$ the level set 
$$
[S]_{m} = S \cap \pi^{-1}(m) = S \cap \left(\NN^{d} \times \{m\}\right).
$$
 
	\end{setup}

\subsection{Mixed multiplicities of $\mm$-primary graded families of monomial ideals}

In this subsection, we prove the existence of mixed multiplicities of graded families of monomial $\mm$-primary ideals in a polynomial ring. 
This extends the main result from \cite{cutkosky2019} in the setting of monomial ideals.
Here our proof depends directly on the Minkowski's Theorem (\autoref{thm_Minkowski_formula}).

We begin by introducing the following  setup.

\begin{setup}\label{setup_mixed_mult_primary}
	Adopt \autoref{setup_mixed_mult}.
	Let $\JJ(1)=\{J(1)_n\}_{n\in \NN}$, $\ldots$, $\JJ(r)=\{J(r)_n\}_{n\in \NN}$ be (not necessarily Noetherian) graded families of $\mm$-primary monomial ideals in $R$.
	Let $c \in \NN$ be a positive integer such that 
	\begin{equation}
	\label{eq_inclus_first_power}
	J(i)_1 \supset \mm^{c-1}  \quad \text{ for all } \quad 1 \le i \le r.
	\end{equation}
	and $c \ge 2$. 
	Thus, it follows that 
	\begin{equation}
	\label{eq_incl_prim_filt}
	J(i)_n \supset \mm^{cn} \; \text{ for all }  \;1 \le i \le r \; \text{ and } \; n \in \NN.
	\end{equation}
\end{setup}

 For a vector $\bn=(n_1,\ldots, n_r)$ in $\NN^{r}$, we  abbreviate $\bJ_\bn=J(1)_{n_1}\cdots J(r)_{n_r}$. 
 We identify each monomial $\bx^\bm = x_1^{m_1}\cdots x_d^{m_d} \in R$ with the corresponding vector $\bm = (m_1,\ldots,m_d) \in \NN^d$.
We now  connect our setting with the information in \autoref{sub_NOB}.  
Let $\bn = (n_1,\ldots,n_r) \in \NN^r$ be an $r$-tuple of non-negative integers.
Thus, for each $m \ge 1$, \autoref{eq_incl_prim_filt} yields the following equation
\begin{equation}
	\label{eq_diff_length_prim}
	\dim_{\kk}\left(R/\bJ_{m\bn}\right) \,=\, \dim_{\kk}\left(R /\mm^{c m\lvert \bn\rvert + 1}\right) - \dim_{\kk}\left(\bJ_{m\bn} /\mm^{c m \lvert \bn \rvert +1} \right). 
\end{equation}
We define the following set
$$
\Gamma_{\bn} :=\Big\{(\bm,m)=(m_1,\ldots, m_d,m)\in \NN^{d+1}\mid \bx^\bm \in \bJ_{m \bn} \mbox{ and } |\bm| \le c m \lvert \bn \rvert  \Big\}.
 $$

 The next lemma provides some basic properties of $\Gamma_{\bn}$. 

\begin{lemma}
	\label{lem_props_Gamma}
	The following statements hold:
	\begin{enumerate}[\rm (i)]
		\item $\Gamma_{\bn}$ is a subsemigroup of the semigroup $\NN^{d+1}$.
		\item $G(\Gamma_{\bn}) = \ZZ^{d+1}$, and so $L(\Gamma_{\bn}) = \RR^{d+1}$.
		\item $(\Gamma_{\bn}, M)$ is a strongly admissible pair, $\dim(\partial M) =d$ and $\ind(\Gamma_{\bn}, M)=\ind(\Gamma_{\bn}, \partial M)=1$.
		\item For any $n \in \NN$ and $1 \le i \le r$, we have 
		$$
		\Delta\left(\Gamma_{n \ee_i}\right) = \left(n \pi_1\left(\Delta\left(\Gamma_{\ee_i}\right)\right), 1\right).
		$$
	\end{enumerate}	
\end{lemma}
\begin{proof}
	(i) Suppose that $(\bm,m),\, (\bm',m') \in \Gamma_{\bn}$, that is, $\bx^\bm \in \bJ_{m\bn}$, $\bx^{\bm'} \in \bJ_{m'\bn}$, $\bm \le cm|\bn|$ and $\bm' \le cm'|\bn|$.
	As $\JJ(1),\ldots,\JJ(r)$ are graded families of ideals, it follows that $\bx^{\bm+\bm'} \in \bJ_{(m+m')\bn}$.
	Thus, the inequality 
	$
	|\bm + \bm'| = |\bm| + |\bm'| \le cm|\bn| + cm'|\bn| = c(m+m')|\bn|
	$ yields the result.

	(ii) By \autoref{eq_inclus_first_power} we can choose $\bm = (m_1,\ldots,m_d) \in \NN^d$ such that $\bx^\bm \in \bJ_\bn$ and $|\bm| = c|\bn| - 1$.
	Since $\bx^{\bm + \ee_i} \in \bJ_\bn$ and $|\bm + \ee_i| = c|\bn|$ for all $1 \le i \le d$, it follows that $\{\ee_1,\ldots,\ee_d\}  \in G(\Gamma_{\bn})$ (here $\ee_i$ denotes the $i$-th elementary basis vector in $\NN^{d+1}$).	
	The equation 
	$$
	\ee_{d+1} = (\bm, 1) - m_1\ee_1 - \cdots - m_d\ee_d
	$$
	implies that $\ee_{d+1} \in G(\Gamma_{\bn})$, and so the result follows.
	
	(iii) The fact that $(\Gamma_{\bn}, M)$ is strongly admissible follows from the way that $\Gamma_{\bn}$ was defined. 
	The other claims are obtained directly from part (ii).
	
	(iv) By definition, for all $m \ge 0$ we have
	$
	\pi_1\left(\big[\Gamma_{n \ee_i}\big]_m\right) = \pi_1\left(\big[\Gamma_{\ee_i}\big]_{nm}\right).
	$	 
	Hence, one obtains 
	$$
	\pi_1\left(\Delta(\Gamma_{n \ee_i})\right) = \pi_1\left(\Con(\Gamma_{n \ee_i}) \cap \pi^{-1}(1)\right) = \pi_1\left(\Con(\Gamma_{\ee_i}) \cap \pi^{-1}(n)\right) = n\pi_1\left(\Delta\left(\Gamma_{\ee_i}\right)\right),
	$$
	and so the result follows.
\end{proof}

Following \cite[\S 1.6]{KAVEH_KHOVANSKII}, we define the {\it levelwise addition } of two subsets $A, B\subseteq \NN^{d+1}=\NN^d\times \NN$ as the set $A\oplus_t B\subseteq \NN^{d+1}$ such that $$\pi_1\left((A\oplus_t B)\cap\pi^{-1}(m)\right)=\pi_1\left(A\cap\pi^{-1}(m)\right)+\pi_1\left(B\cap\pi^{-1}(m)\right),$$
for every $m\in \NN$.  
The following proposition decomposes $\Gamma_{\bn}$ as a levelwise sum of simpler semigroups.
This basic result can be seen as the main step in our proof.

\begin{proposition}
	\label{prop_decomp_Gamma}
	Assume \autoref{setup_mixed_mult_primary}.
	We have the equality 
	$
	\Gamma_{\bn} \,=\,  \Gamma_{n_1\ee_1} \oplus_t \cdots \oplus_t  	\Gamma_{n_r\ee_r}.
	$
\end{proposition}
\begin{proof}
	The result is obtained from \autoref{prop_decomp_Gamma_non_mm}(ii) and the fact that $\beta(J(i)_n) \le cn$ for all $1 \le i \le r$ and $n \in \NN$ (see the assumptions and notations  in \autoref{setup_mixed_mult_general}).
\end{proof}

From \autoref{eq_diff_length_prim}, and the fact that $\JJ(1),\ldots,\JJ(r)$ are graded families of monomial ideals, we obtain that 
\begin{align}
	\label{eq_len_equal_num_points}
	\begin{split}
	\dim_{\kk}\left(R/\bJ_{m\bn}\right) &= \dim_{\kk}\left(R /\mm^{c m\lvert \bn \rvert+1}\right) - \dim_{\kk}\left(\bJ_{m\bn} /\mm^{c m \lvert \bn \rvert+1} \right).\\
	&= \binom{cm|\bn|+d}{d} \;-\; \#\big[\Gamma_{\bn}\big]_m.
	\end{split}
\end{align}

After the previous preparatory results, we are ready for the main result of this subsection. 
The following theorem shows the existence of a homogeneous  polynomial that can be used to define the mixed multiplicities of the graded families $\JJ(1), \ldots, \JJ(r)$.
 As a consequence of the proof, we describe the coefficients of the polynomial explicitly in terms of the mixed volumes of certain Newton-Okounkov bodies.

\begin{theorem}
	\label{thm_m_primary_case}
	Assume \autoref{setup_mixed_mult_primary}.
	The function 
	$$
	F(n_1,\ldots,n_r) = \lim_{m\to \infty} \frac{\dim_\kk\big(R/J(1)_{mn_1}\cdots J(r)_{mn_r})\big)}{m^d}
	$$
	is equal to a homogeneous polynomial $G(\bn)=G(n_1,\ldots,n_r)$ of total degree $d$ with real coefficients for all $\bn = (n_1,\ldots,n_r) \in \NN^r$.
	Explicitly, the polynomial $G(\bn)$ is given by 
	$$
	G(\bn) = \sum_{|\dd| = d} \frac{1}{\dd!} \left( c^d - \MV_d\left({\Delta(\Gamma)}_{\dd}\right) \right)\, \bn^\dd,
	$$
	where $\Delta(\Gamma)$ denotes the sequence of Newton-Okounkov bodies $\Delta(\Gamma) = (\Delta(\Gamma_{\ee_1}),\ldots,\Delta(\Gamma_{\ee_r}))$.
\end{theorem}
\begin{proof}
	Let $\bn = (n_1,\ldots,n_r) \in \NN^r$.
	By using \autoref{thm_limit_KK}, \autoref{lem_props_Gamma}(iii) and \autoref{eq_len_equal_num_points} we obtain the equation 
	\begin{equation}
		\label{eq_F_as_vol}
		F(\bn) \,=\, \lim_{m\to \infty} \frac{\binom{cm|\bn|+d}{d}}{m^d} - \lim_{m\to \infty} \frac{\#\big[\Gamma_{\bn}\big]_m}{m^d} = \frac{c^d|\bn|^d}{d!} - \Vol_{d}\left(\Delta(\Gamma_\bn)\right).
	\end{equation}
	Due to \autoref{prop_decomp_Gamma}, \autoref{lem_props_Gamma}(iv) and \cite[Proposition 1.32]{KAVEH_KHOVANSKII} we get the equality 
	$$
	\pi_1\left(\Delta(\Gamma_{\bn})\right) = \pi_1\left(\Delta(\Gamma_{n_1\ee_1})\right) + \cdots +	\pi_1\left(\Delta(\Gamma_{n_r\ee_r})\right) = n_1\pi_1\left(\Delta(\Gamma_{\ee_1})\right) + \cdots +	n_r\pi_1\left(\Delta(\Gamma_{\ee_r})\right).
	$$
	Thus, \autoref{eq_volume_poly} implies that 
	\begin{align}
		\label{eq_Vol_gamma_n}
		\begin{split}		
		\Vol_{d}\left(\Delta(\Gamma_\bn)\right) &= \Vol_{d}\left(\pi_1(\Delta(\Gamma_\bn))\right)\\
		&= \sum_{\lvert \dd \rvert=d} \frac{1}{\dd!}\MV_d({\pi_1(\Delta(\Gamma))}_\dd)  \bn^\dd = \sum_{\lvert \dd \rvert=d} \frac{1}{\dd!}\MV_d({\Delta(\Gamma)}_\dd)  \bn^\dd
		\end{split}
	\end{align}
	where $\pi_1(\Delta(\Gamma))$ denotes the sequence  $\pi_1(\Delta(\Gamma)) = (\pi_1(\Delta(\Gamma_{\ee_1})),\ldots,\pi_1(\Delta(\Gamma_{\ee_r})))$ of convex bodies.
	Finally, the result follows by combining \autoref{eq_F_as_vol} and \autoref{eq_Vol_gamma_n}.
\end{proof}

\begin{proposition}
	\label{lem_non_neg_mm_prim}
	Assume \autoref{setup_mixed_mult_primary} and use the same notation of \autoref{thm_m_primary_case}. 
	Let $\dd = (d_1,\ldots,d_r) \in \NN^r$ with $\lvert \dd \rvert = d$. 
	Then, one has that $c^d - \MV_d({\Delta(\Gamma)}_{\dd}) \ge 0$.
\end{proposition}
\begin{proof}
	Let $\Sigma \subset \RR^d$ be the polytope given as the convex hull of the points $\mathbf{0}, c \ee_1, \ldots, c \ee_d \subset \RR^d$.
	Consider the polytope $\Delta = \Sigma \times \{1\} \subset \RR^{d+1}$, and notice that by construction we have $\Delta \supset \Delta(\Gamma_{\ee_i})$ for all $1 \le i \le r$.
	Since $\Vol_{d}(\Delta) = c^d/d!$ and $\MV_{d}(\Delta,\ldots,\Delta) = d!\Vol_{d}(\Delta)$, the inequality 
	\begin{align*}
	  c^d - \MV_d({\Delta(\Gamma)}_{(d_1,\ldots,d_r)})  
	= \MV_{d}(\Delta,\ldots,\Delta) - \MV_d({\Delta(\Gamma)}_{(d_1,\ldots,d_r)})  \ge 0
	\end{align*}
	follows from the monotonicity of mixed volumes (see, e.g., \cite[Eq. 5.25]{SCHNEIDER}).
\end{proof}

With \autoref{thm_m_primary_case} in hand we are ready to define the mixed multiplicities of graded families of $\mm$-primary monomial ideals.
Due to \autoref{lem_non_neg_mm_prim}, the mixed multiplicities defined below are always non-negative.

\begin{definition}\label{def_mix_mult_prim}
	Assume \autoref{setup_mixed_mult_primary} and let $G(\bn)$ be as in \autoref{thm_m_primary_case}. 
	Write 
	$$
	G(\bn) = \sum_{|\dd| = d} \frac{1}{\dd!}\, e_{\dd}(\JJ(1),\ldots,\JJ(r))\, \bn^\dd.
	$$
	For each $\dd = (d_1,\ldots,d_r) \in \NN^r$ with $|\dd|=d$, we define the non-negative real number $$e_{\dd}(\JJ(1),\ldots,\JJ(r)) \ge 0$$ to be the {\it mixed multiplicities of type $\dd$ of  $\JJ(1) = \{J(1)_n\}_{n\in \NN}$,  $\ldots$, $\JJ(r) = \{J(r)_n\}_{n\in \NN}$}.	
\end{definition}

\subsection{Mixed multiplicities of arbitrary graded families of monomial ideals}
In this subsection, we introduce the notion of mixed multiplicities for arbitrary graded families of monomial ideals under mild conditions.
We  begin with the  following   setup that is used in our results. 

\begin{setup}
	\label{setup_mixed_mult_general}
	Adopt \autoref{setup_mixed_mult}.
	Let $\II = \{I_n\}_{n \in \NN}$ be a (not necessarily Noetherian) graded family of $\mm$-primary monomial ideals.
	Let $\JJ(1)=\{J(1)_n\}_{n\in \NN}$, $\ldots$, $\JJ(r)=\{J(r)_n\}_{n\in \NN}$ be (not necessarily Noetherian) graded families of monomial ideals in $R$.

	For a homogeneous ideal $J$ we denote $\beta(J)=\max\{j\mid [J\otimes_R \mathbb{k}]_j\neq 0\}$, that is, the maximum degree of a minimal set of homogeneous generators of $J$. We assume that there exists $\beta \in \NN$ satisfying 
	$$
 \beta(J(i)_n) \le \beta n	$$ 
	for all $1 \le i \le r$ and $n \in \NN$; similar assumptions have been considered in previous works regarding limits of graded families of ideals \cite[Theorem 6.1]{cutkosky2014}.
\end{setup}

We have the following simple observation that plays an important role in our approach.

\begin{lemma}
	\label{lem_eq_filt}
	There exists an integer $c > \beta$ such that 
	$$
	\mm^{c(n_0+|\bn|)} \, \cap \, \bJ_\bn \;=\; \mm^{c(n_0+|\bn|)} \, \cap \, I_{n_0}\bJ_\bn
	$$
	for all $n_0 \in \NN$ and $\bn = (n_1,\ldots,n_r) \in \NN^r$.
\end{lemma}
\begin{proof}
	Let $c' \in \NN$ be a positive integer such that $I_1 \supset \mm^{c'}$; in particular, $I_{n_0} \supset \mm^{n_0c'}$.
	Let $c = \max\{\beta+1, c'\}$.
	Since $\beta(\bJ_\bn) \le \beta|\bn|$, we obtain the following inclusion
	$$
	\mm^{c(n_0+|\bn|)} \, \cap \, I_{n_0}\bJ_\bn \;\supset\; \mm^{c(n_0+|\bn|)} \, \cap \, \mm^{cn_0}\bJ_\bn \;=\; \mm^{c(n_0+|\bn|)} \, \cap \, \bJ_\bn,
	$$
 the result follows.
\end{proof}

Let $c$ be as in \autoref{lem_eq_filt}, $n_0 \in \NN$, and $\bn = (n_1,\ldots,n_r) \in \NN^r$.
We define the sets
\begin{equation}
	\label{eq_def_Gamma}
	\Gamma_{n_0,\bn} :=\Big\{(\bm,m)=(m_1,\ldots, m_d,m)\in \NN^{d+1}\mid \bx^\bm \in \bJ_{m \bn} \mbox{ and } |\bm| \le c m (n_0+\lvert \bn \rvert)  \Big\}
\end{equation}
and 
\begin{equation}
	\label{eq_def_Gamma_hat}
	\widehat{\Gamma}_{n_0,\bn} :=\Big\{(\bm,m)=(m_1,\ldots, m_d,m)\in \NN^{d+1}\mid \bx^\bm \in I_{mn_0}\bJ_{m \bn} \mbox{ and } |\bm| \le c m (n_0+\lvert \bn \rvert)  \Big\}.
\end{equation}

The lemma below is an equivalent to \autoref{lem_props_Gamma} and its proof follows verbatim.

\begin{lemma}
	\label{lem_props_Gamma_non_mm_prim}
	Let $S \subset \NN^{d+1}$ be equal to either $\Gamma_{n_0,\bn}$ or $\widehat{\Gamma}_{n_0,\bn}$.
	The following statements hold:
	\begin{enumerate}[\rm (i)]
		\item $S$ is a subsemigroup of the semigroup $\NN^{d+1}$.
		\item $G(S) = \ZZ^{d+1}$, and so $L(S) = \RR^{d+1}$.
		\item $(S, M)$ is a strongly admissible pair, $\dim(\partial M) =d$ and $\ind(S, M)=\ind(S, \partial M)=1$.
		\item For any $n \in \NN$ and $1 \le i \le r$, we have 
		$$
		\Delta\big(\Gamma_{0, n \ee_i}\big) = \left(n \pi_1\left(\Delta\big(\Gamma_{0,\ee_i}\big)\right), 1\right) \; \text{ and } \; \Delta\big(\widehat{\Gamma}_{0, n \ee_i}\big) = \big(n \pi_1\big(\Delta\big(\widehat{\Gamma}_{0,\ee_i}\big)\big), 1\big).
		$$
		\item For any $n \in \NN$, $\Delta\big(\Gamma_{n, \mathbf{0}}\big) = \left(n \pi_1\big(\Delta\big(\Gamma_{1,\mathbf{0}}\big)\big), 1\right)$  and $\Delta\big(\widehat{\Gamma}_{n, \mathbf{0}}\big) = \big(n \pi_1\big(\Delta\big(\widehat{\Gamma}_{1,\mathbf{0}}\big)\big), 1\big)$.
	\end{enumerate}	
\end{lemma}

The next proposition decomposes $\Gamma_{n_0,\bn}$ and $\widehat{\Gamma}_{n_0,\bn}$ as the levelwise sum of simpler semigroups (this result plays the same role that \autoref{prop_decomp_Gamma} played in the previous subsection).

\begin{proposition}
	\label{prop_decomp_Gamma_non_mm}
	Assume \autoref{setup_mixed_mult_general}.
	We have the following equalities:
	\begin{enumerate}[(i)]
		\item $
		\Gamma_{n_0,\bn} \,=\,   \Gamma_{n_0,\mathbf{0}} \oplus_t \Gamma_{0,n_1\ee_1} \oplus_t \cdots \oplus_t  	\Gamma_{0,n_r\ee_r}.
		$
		\item $
		\widehat{\Gamma}_{n_0,\bn} \,=\,   \widehat{\Gamma}_{n_0,\mathbf{0}} \oplus_t \widehat{\Gamma}_{0,n_1\ee_1} \oplus_t \cdots \oplus_t  	\widehat{\Gamma}_{0,n_r\ee_r}.
		$
	\end{enumerate}
\end{proposition}
\begin{proof}
	(ii)
	For each $m \ge 0$, we need to show that 
	$$
	\pi_1([\widehat{\Gamma}_{n_0,\bn}]_m) \,=\, 
	\pi_1([\widehat{\Gamma}_{n_0,\mathbf{0}}]_m) +  \pi_1([\widehat{\Gamma}_{0,n_1\ee_1}]_m) + \cdots +	\pi_1([\widehat{\Gamma}_{0,n_r\ee_r}]_m).
	$$
	
	Fix $m \in \ZZ_{>0}$.
	
	First, we concentrate on the inclusion ``$\supset$''.
	Let $w_0 \in \big[\widehat{\Gamma}_{n_0,\mathbf{0}}\big]_m$ and, 
	for each $1 \le i \le r$, let $w_i \in \big[\widehat{\Gamma}_{0,n_i\ee_i}\big]_m$.
	Note that, for each $1 \le i \le r$, there exists $\bx^{\bm_i} \in J(i)_{mn_i}$ such that $w_i = (\bm_i, m) \in \NN^{d+1}$. 
	Similarly, there exists $\bx^{\bm_0} \in I_{mn_0}$ such that $w_0 = (\bm_0, m) \in \NN^{d+1}$. 
	Since $\lvert \bm_i \rvert \le c m n_i$ for $0 \le i \le r$, it is clear that 
	$
	\bx^{\bm_0}\cdots \bx^{\bm_r} \in I_{mn_0}\bJ_{m \bn}$ and  $\lvert \bm_0+\cdots+ \bm_r \rvert = \lvert \bm_0\rvert	+ \cdots + \lvert \bm_r \rvert	\le c m (n_0+\lvert \bn \rvert).
	$
	Therefore, it follows that $\widehat{\Gamma}_{n_0,\bn} \supset \widehat{\Gamma}_{n_0,\mathbf{0}} \oplus_t \widehat{\Gamma}_{0,n_1\ee_1} \oplus_t \cdots \oplus_t  	\widehat{\Gamma}_{0,n_r\ee_r}$.
	
	\smallskip
	
	Next, we focus on the inclusion ``$\subset$''.
	Let $w \in \big[\widehat{\Gamma}_{n_0,\bn}\big]_m$.
	Since $\II, \JJ(1),\ldots,\JJ(r)$ are graded families of monomial ideals, there exist $\bx^{\bm_0} \in I_{mn_0}$ and $\bx^{\bm_i} \in J(i)_{mn_i}$ such that $w = (\bm_0+\cdots +\bm_r, m) \in \NN^{d+1}$.

	By assumption we have $\sum_{i = 0}^r \lvert \bm_i \rvert \le c m (n_0+\lvert \bn \rvert)$.
	For ease of notation, set $J(0)_n = I_n$ for all $n \in \NN$.
	
	Let $l(\bm_0,\ldots,\bm_r) := \sum_{i = 0}^r \max\{ \lvert \bm_i \rvert - c m n_i,\, 0 \}$.
	If $l(\bm_0,\ldots,\bm_r) = 0$, it then follows that $\lvert \bm_i \rvert \le c m n_i$ for $0 \le i \le r$, and so we obtain that $\pi_1(w) = \pi_1(w_0) + \cdots + \pi_1(w_r)$ where $w_0 = (\bm_0, m) \in \big[\widehat{\Gamma}_{n_0,\mathbf{0}}\big]_m$ and $w_i = (\bm_i, m) \in \big[\widehat{\Gamma}_{0,n_i\ee_i}\big]_m$ for $1 \le i \le r$.
	On the other hand, suppose that $l(\bm_0,\ldots,\bm_r) > 0$.
	Thus, there exist $0 \le j_1,j_2 \le r$ such that $\lvert \bm_{j_1} \rvert > c m n_{j_1}$ and $\lvert \bm_{j_2} \rvert < c m n_{j_2}$.
	From the fact that  $\beta(J(j_1)_{mn_{j_1}}) \le cmn_{j_1}$, we can choose $1 \le k \le d$ such that $\bx^{\bm_{j_1}-\ee_k} \in J(j_1)_{mn_{j_1}}$.
	For $0 \le i \le r$, we now set
	$$
	\bx^{\bm_i'} \in J(i)_{mn_i} \qquad \text{ by } \qquad  \bm_i' = \begin{cases}
	\bm_{j_1} - \ee_k \quad \text{ if } i = j_1\\
	\bm_{j_2} + \ee_k \quad \text{ if } i = j_2\\
	\bm_i \quad\quad\;\,\,\;\;\,\, \text{ otherwise.}\\
	\end{cases}
	$$
	Notice that $\pi_1(w) = \bm_0'+\cdots+\bm_r'$ and $l(\bm_0',\ldots,\bm_r')=l(\bm_0,\ldots,\bm_r)-1$.
	Therefore, by inducting on $l(\bm_1,\ldots,\bm_r)$, we obtain the other inclusion $\widehat{\Gamma}_{n_0,\bn} \subset \widehat{\Gamma}_{n_0,\mathbf{0}} \oplus_t \widehat{\Gamma}_{0,n_1\ee_1} \oplus_t \cdots \oplus_t  	\widehat{\Gamma}_{0,n_r\ee_r}$.
	
	\smallskip
	
	(i) This part follows similarly, for example by following the arguments of  part (ii) with  $I_n = R$ for all $n \in \NN$.
\end{proof}

From \autoref{lem_eq_filt} and the fact that $\II, \JJ(1), \ldots, \JJ(r)$ are graded families of monomial ideals, we obtain the following equalities
\begin{align}
	\label{eq_diff_gen_case}
	\begin{split}
	\dim_{\kk}\left(\bJ_{m\bn}/I_{mn_0}\bJ_{m\bn}\right) &= \dim_{\kk}\left(\bJ_{m\bn} / \left(\mm^{cm(n_0+|\bn|)+1} \, \cap \, \bJ_{m\bn}\right) \right)\\
	 & \qquad\qquad - 
	\dim_{\kk}\left(I_{mn_0}\bJ_{m\bn} / \left(\mm^{cm(n_0+|\bn|)+1} \, \cap \, I_{mn_0}\bJ_{m\bn}\right) \right)\\
	&= \#\big[\Gamma_{n_0,\bn}\big]_m \;-\; \#\big[\widehat{\Gamma}_{n_0,\bn}\big]_m.
	\end{split}	
\end{align}

We are now ready for the main result of this section. 
We show the existence of a homogeneous polynomial that allows us to define the mixed multiplicities of the graded families $\II, \JJ(1), \ldots, \JJ(r)$.
Additionally, we explicitly describe this polynomial in terms of the mixed volume of certain Newton-Okounkov bodies.

\begin{theorem}
	\label{thm_general_case}
	Assume \autoref{setup_mixed_mult_general}.
	The function 
	$$
	F(n_0,n_1,\ldots,n_r) = \lim_{m\to \infty} \frac{\dim_\kk\big(J(1)_{mn_1}\cdots J(r)_{mn_r}\,\big/\,I_{mn_0}J(1)_{mn_1}\cdots J(r)_{mn_r}\big)}{m^d}
	$$
	is equal to a homogeneous polynomial $G(n_0,\bn) = G(n_0,n_1,\ldots,n_r)$ of total degree $d$ with real coefficients for all $n_0 \in \NN$ and $\bn = (n_1,\ldots,n_r) \in \NN^r$.
	Explicitly, the polynomial $G(n_0,\bn)$ is given by 
	$$
		G(n_0,\bn) = \sum_{d_0+|\dd| = d} \frac{1}{d_0!\dd!} \,\Bigg( \MV_d\left({\Delta(\Gamma)}_{(d_0,\dd)}\right)
		 - \MV_d\left({\Delta(\widehat{\Gamma})}_{(d_0,\dd)}\right) \Bigg)\, n_0^{d_0}\bn^\dd,
	$$
	where $\Delta(\Gamma)$ and $\Delta(\widehat{\Gamma})$ denote the sequences of Newton-Okounkov bodies $$
	\Delta(\Gamma) = \left(\Delta(\Gamma_{1,\mathbf{0}}),\Delta(\Gamma_{0,\ee_1}),\ldots,\Delta(\Gamma_{0,\ee_r})\right) \quad \text{ and } \quad \Delta(\widehat{\Gamma}) = \left(\Delta(\widehat{\Gamma}_{1,\mathbf{0}}),\Delta(\widehat{\Gamma}_{0,\ee_1}),\ldots,\Delta(\widehat{\Gamma}_{0,\ee_r})\right),
	$$
	respectively.
\end{theorem}
\begin{proof}
	The proof follows along the same lines of \autoref{thm_m_primary_case}.
	Let $n_0 \in \NN$ and $\bn = (n_1,\ldots,n_r) \in \NN^r$.
	By using \autoref{thm_limit_KK}, \autoref{lem_props_Gamma_non_mm_prim}(iii) and \autoref{eq_diff_gen_case} we obtain the equation 
	\begin{equation}
	\label{eq_F_as_vols_general}
	F(n_0,\bn) \,=\, \lim_{m\to \infty} \frac{\#\big[\Gamma_{n_0,\bn}\big]_m}{m^d} - \lim_{m\to \infty} \frac{\#\big[\widehat{\Gamma}_{n_0,\bn}\big]_m}{m^d} = \Vol_{d}\left(\Delta(\Gamma_{n_0,\bn})\right) - \Vol_{d}\left(\Delta(\widehat{\Gamma}_{n_0,\bn})\right).
	\end{equation}
	From \autoref{prop_decomp_Gamma_non_mm}, \autoref{lem_props_Gamma_non_mm_prim}(iv)(v), \cite[Proposition 1.32]{KAVEH_KHOVANSKII} and \autoref{eq_volume_poly} we obtain that 
	$$
	\Vol_{d}\left(\Delta(\Gamma_{n_0,\bn})\right)  = \sum_{d_0+\lvert \dd \rvert=d} \frac{1}{d_0!\dd!}\MV_d\left({\Delta(\Gamma)}_{(d_0,\dd)}\right)  n_0^{d_0}\bn^\dd
	$$
	 and  
	$$
	\Vol_{d}\left(\Delta(\widehat{\Gamma}_{n_0,\bn})\right)  = \sum_{d_0+\lvert \dd \rvert=d} \frac{1}{d_0!\dd!}\MV_d\left({\Delta(\widehat{\Gamma})}_{(d_0,\dd)}\right)  n_0^{d_0}\bn^\dd.
	$$
	So, the result follows.
\end{proof}

\begin{lemma}
	\label{lem_non_neg_general}
	Assume \autoref{setup_mixed_mult_general} and use the same notation of \autoref{thm_general_case}. 
	Let $d_0 \in \NN$ and $\dd = (d_1,\ldots,d_r) \in \NN^r$ with $d_0 + \lvert \dd \rvert = d$. 
	Then:
	\begin{enumerate}[\rm (i)]
		\item $\MV_d({\Delta(\Gamma)}_{(d_0,\dd)})
		- \MV_d({\Delta(\widehat{\Gamma})}_{(d_0,\dd)}) \ge 0$.
		\item  $\MV_d({\Delta(\Gamma)}_{(d_0,\dd)})
		- \MV_d({\Delta(\widehat{\Gamma})}_{(d_0,\dd)}) = 0$ when $d_0=0$.
	\end{enumerate}
\end{lemma}
\begin{proof}
	Notice that $\Delta(\Gamma_{1,\mathbf{0}}) \supset \Delta(\widehat{\Gamma}_{1,\mathbf{0}})$ and that $\Delta(\Gamma_{0,\ee_i}) = \Delta(\widehat{\Gamma}_{0,\ee_i})$ for all $1 \le i \le r$.
	So, the result follows from the monotonicity of mixed volumes (see, e.g., \cite[Eq. 5.25]{SCHNEIDER}).
\end{proof}

After proving \autoref{thm_general_case} we can define the mixed multiplicities of graded families of monomial ideals.
As a consequence of \autoref{lem_non_neg_general}, these mixed multiplicities are always non-negative and we can restrict ourselves to the terms of the form $n_0^{d_0+1}\bn^\dd$ in the  definition below.

\begin{definition}\label{def_mix_mult_general}
	Assume \autoref{setup_mixed_mult_general} and let $G(n_0, \bn)$ be as in \autoref{thm_general_case}. 
	Write 
	$$
	G(n_0,\bn) = \sum_{d_0+|\dd| = d-1} \frac{1}{(d_0+1)!\dd!}\, e_{(d_0,\dd)}\left(\II|\JJ(1),\ldots,\JJ(r)\right)\, n_0^{d_0+1}\bn^\dd.
	$$
	For each $d_0 \in \NN$ and $\dd = (d_1,\ldots,d_r) \in \NN^r$ with $d_0+|\dd| = d-1$, we define the non-negative real number $$e_{(d_0,\dd)}\left(\II|\JJ(1),\ldots,\JJ(r)\right)\ge 0$$ to be the {\it mixed multiplicity of type $(d_0,\dd)$ of $\JJ(1) = \{J(1)_n\}_{n\in \NN}$,  $\ldots$, $\JJ(r) = \{J(r)_n\}_{n\in \NN}$ with respect to $\II = \{I_n\}_{n\in \NN}$}.
\end{definition}

The following remark shows that \autoref{def_mix_mult_prim} and \autoref{def_mix_mult_general} agree in the $\mm$-primary case.

\begin{remark}
	Assume \autoref{setup_mixed_mult_general} and suppose that $\JJ(1),\ldots,\JJ(r)$ are also graded families of $\mm$-primary monomial ideals. 
	For all $m, n_0 \in \NN$ and $\bn = (n_1,\ldots,n_r) \in \NN^r$, we have the short exact sequence 
	$$
	0 \rightarrow  \bJ_{m\bn}/I_{mn_0}\bJ_{m\bn} \rightarrow R/I_{mn_0}\bJ_{m\bn} \rightarrow R/\bJ_{m\bn} \rightarrow 0.
	$$
	So, for each $d_0 \in \NN$ and $\dd = (d_1,\ldots,d_r) \in \NN^r$ with $d_0+|\dd| = d$, we can deduce the following:
	\begin{enumerate}[\rm (i)]
		\item If $d_0 = 0$, then $e_{(d_0,\dd)}(\II,\JJ(1),\ldots,\JJ(r)) = e_{\dd}(\JJ(1),\ldots,\JJ(r))$.
		\item If $d_0 > 0$, then $e_{(d_0,\dd)}(\II,\JJ(1),\ldots,\JJ(r)) = e_{(d_0-1,\dd)}(\II\mid\JJ(1),\ldots,\JJ(r))$.
	\end{enumerate}
\end{remark}

\section{A ``Volume = Multiplicity formula" for mixed multiplicities}
\label{subsect_Vol_eq_Mult}

In this section, we focus on proving \autoref{thmB} (see \autoref{thm_mult_eq_vol}) which gives a ``Volume = Multiplicity formula'' for mixed multiplicities.
This can be seen as an extension of the usual ``Volume = Multiplicity formula'' for graded families of ideals (see, e.g., \cite[Theorem 6.5]{cutkosky2014}).
Before that, we need to briefly recall the notion of mixed multiplicities for the case of ideals (for more details, see, e.g., \cite{TRUNG_VERMA_MIXED_VOL}).

Throughout this subsection we adopt \autoref{setup_mixed_mult_general} and the following extra piece of notation.

\begin{notation}\label{not_the_p}
Assume \autoref{setup_mixed_mult_general}. For every $p \in \NN$ and $\bn = (n_1,\ldots,n_r) \in \NN^r$, let $\bJ(p)^\bn$ denote the ideal $J(1)_p^{n_1}\cdots J(r)_p^{n_r}$.
\end{notation}

Let $I \subset R$ be a homogeneous $\mm$-primary ideal and $J_1,\ldots,J_r \subset R$ be homogeneous ideals.
Since $I$ is $\mm$-primary, we have that
\begin{equation}
	\label{eq_alg_mixed_mult}
	T = T(I\mid J_1,\ldots,J_r) \;:=\; \bigoplus_{n_0 \ge 0, n_1 \ge 0 \ldots,n_r \ge 0} I^{n_0}J_1^{n_1}\cdots J_r^{n_r} \big/ I^{n_0+1}J_1^{n_1}\cdots J_r^{n_r}
\end{equation}
is a finitely generated standard $\NN^{r+1}$-graded algebra over the Artinian local ring $R/I$. 
From \cite[Theorem 1.2(a)]{TRUNG_VERMA_MIXED_VOL}, one has a polynomial $P_T(n_0,n_1,\ldots,n_r) \in \QQ[n_0,n_1,\ldots,n_r]$ of degree $d-1=\dim(R)-1$ such that  $P_T(\nu) = \dim_{\kk}\left([T]_{\nu}\right)$ for all $\nu \in \NN^{r+1}$ with $\nu \gg \mathbf{0}$.
Furthermore, if we write 
\begin{equation}
	\label{eq_poly_mixed_mult_ideals}
P_{T}(n_0,n_1,\ldots,n_r) = \sum_{d_0,d_1,\ldots,d_r \ge 0} e(d_0,d_1,\ldots,d_r)\binom{n_0+d_0}{d_0}\binom{n_1+d_1}{d_1}\cdots \binom{n_r+d_r}{d_r},
\end{equation}
then $0 \le e(d_0,d_1,\ldots,d_r) \in \ZZ$ for all $d_0+d_1+\cdots+d_r = d-1$.
For each $d_0 \in \NN$ and $\dd =(d_1,\ldots,d_r) \in \NN^{r}$ with $d_0 + \lvert\dd \rvert = d-1$, we say that
\begin{equation}
	\label{eq_mixed_mult_ideals}
	e_{(d_0,\dd)}\left(I\mid J_1,\ldots,J_r\right) \;:= \; e(d_0,d_1,\ldots,d_r)
\end{equation}
is the \emph{mixed multiplicity of type $(d_0,\dd)$ of $J_1,\ldots,J_r$ with respect to $I$}.
The following lemma shows that the definition given in \autoref{eq_mixed_mult_ideals} agrees with the one given in \autoref{def_mix_mult_general}.

\begin{lemma}
	\label{lem_eq_mixed_mult_two_defs}
	Let $I \subset R$ be a monomial $\mm$-primary ideal and $J_1,\ldots,J_r \subset R$ be monomial ideals.
	Consider the monomial filtrations $\{I^n\}_{n \in \NN}, \{J_1^n\}_{n \in \NN}, \ldots, \{J_r^n\}_{n \in \NN}$ given by the powers of $I,J_1,\ldots,J_r$.
	Then, we have the equality 
	$$
	e_{(d_0,\dd)}\big(\{I^n\}_{n \in \NN}\mid \{J_1^n\}_{n \in \NN},\ldots,\{J_r^n\}_{n \in \NN}\big) = e_{(d_0,\dd)}\left(I\mid J_1,\ldots,J_r\right)
	$$
	for each $d_0 \in \NN$ and $\dd =(d_1,\ldots,d_r) \in \NN^{r}$ with $d_0 + \lvert\dd \rvert = d-1$.
\end{lemma}
\begin{proof}
	Let $F(n_0,\bn)$ be the function $F(n_0,\bn) = \lim_{m\to \infty} \frac{\dim_\kk\big(\bJ^{m\bn}\,\big/\,I^{mn_0}\bJ^{m\bn}\big)}{m^d}$ where $\bn = (n_1,\ldots,n_r)$ and $\bJ^{m\bn}$ denotes the ideal $\bJ^{m\bn}=J_1^{mn_1}\cdots J_r^{mn_r} \subset R$.
	
	For each $n_0 \in \NN$ and $\bn \in \NN^r$ we have the following equality 
	$$
	\dim_\kk\big(\bJ^{\bn}\,\big/\,I^{n_0}\bJ^{\bn}\big) = \sum_{k=0}^{n_0-1} \dim_{\kk}\left(I^k\bJ^{\bn}\,\big/\,I^{k+1}\bJ^{\bn}\right) = \sum_{k=0}^{n_0-1} \dim_{\kk}\left([T]_{(k,\bn)}\right)
	$$
	(where $T$ is the algebra introduced in \autoref{eq_alg_mixed_mult}).
	Let $\nu = (\nu_0,\ldots, \nu_{r}) \in \NN^{r+1}$ such that $P_T(n_0,\bn) = \dim_{\kk}\left([T]_{(n_0,\bn)}\right)$ for all $(n_0,\bn) \ge \nu$.
	Thus, for all $(n_0,\bn) \ge \nu$, we can write 
	\begin{equation}
		\label{eq_quot_as_consecutive_quots}
		\dim_\kk\big(\bJ^{\bn}\,\big/\,I^{n_0}\bJ^{\bn}\big) = \sum_{k=0}^{\nu_0-1}\dim_{\kk}\left([T]_{(k,\bn)}\right) + \sum_{k=\nu_1}^{n_0-1} P_T(k,\bn).
	\end{equation}

	For any $k \in \NN$, one has that $[T]_{(k,*,\ldots,*)} = \bigoplus_{\bn \ge \mathbf{0}} I^{k}\bJ^{\bn} \big/ I^{k+1}\bJ^{\bn}$ is a finitely generated $\NN^r$-graded module over the finitely generated standard $\NN^r$-graded algebra $[T]_{(0,*,\ldots,*)} = \bigoplus_{\bn \ge \mathbf{0}} \bJ^{\bn} \big/ I\bJ^{\bn}$.
	From \cite[Theorem 4.1]{HERMANN_MULTIGRAD} (also, see \cite[Theorem 3.4]{MIXED_MULT}), for all $\bn \gg \mathbf{0}$, we obtain that 
	$$
	\dim_{\kk}\left([T]_{(k,\bn)}\right) = P_{[T]_{(k,*,\ldots,*)}}(\bn) \quad\text{ 	for some polynomial } \quad P_{[T]_{(k,*,\ldots,*)}}(\bn)
	$$ 
  with degree bounded by  $\dim\left(\multProj\left([T]_{(0,*,\ldots,*)}\right)\right)$.
	Since $I$ is an $\mm$-primary ideal, we have the equality $\dim\left(\multProj\left([T]_{(0,*,\ldots,*)}\right)\right) = \dim\left(\multProj\left(\FF(J_1,\ldots,J_r)\right)\right)$, where $\FF(J_1,\ldots,J_r)$ denotes the special fiber ring $\FF(J_1,\ldots,J_r) = \Rees(J_1,\ldots,J_r) \otimes_R R/\mm$.
	By using the Segre embedding we get the isomorphism 
	$$
	\multProj\big(\Rees(J_{1},\ldots,J_{r}) \otimes_R R/\mm\big) \cong \Proj\big(\Rees(J_{1}\cdots J_{r}) \otimes_R R/\mm\big).
	$$
	Therefore, for all $\bn \gg \mathbf{0}$, we obtain that $\dim_{\kk}\left([T]_{(k,\bn)}\right) = P_{[T]_{(k,*,\ldots,*)}}(\bn)$ and 
	\begin{equation}
		\label{eq_upper_bound_deg_pol}
		 \deg\left(P_{[T]_{(k,*,\ldots,*)}}(\bn) \right) \le \dim\left(\Proj\big(\Rees(J_{1}\cdots J_{r}) \otimes_R R/\mm\big)\right) = \ell(J_1\cdots J_r) -1 \le d-1,
	\end{equation}
	where $\ell(J_1\cdots J_r)$ denotes the analytic spread of $J_1\cdots J_r$ and the last inequality follows from \cite[Proposition 5.1.6]{huneke2006integral}.

	By combining \autoref{eq_quot_as_consecutive_quots} and \autoref{eq_upper_bound_deg_pol}, we obtain the following equality
	\begin{align*}
		F(n_0,\bn) = \lim_{m\to \infty} \frac{\dim_\kk\big(\bJ^{m\bn}\,\big/\,I_{mn_0}\bJ^{m\bn}\big)}{m^d} &= \lim_{m\to \infty} \frac{
			\sum_{k=0}^{\nu_0-1}\dim_{\kk}\left([T]_{(k,m\bn)}\right) + \sum_{k=\nu_1}^{mn_0-1} P_T(k,m\bn)
		}{m^d} \\
		&= \lim_{m\to \infty}\frac{\sum_{k=\nu_1}^{mn_0-1} P_T(k,m\bn)}{m^d}. 
	\end{align*}	
	Since $\deg(P_T(n_0,\bn))=d-1$, we can write $F(n_0,\bn)= \lim_{m\to \infty}\frac{\sum_{k=0}^{mn_0-1} P_T(k,m\bn)}{m^d}$.
	Notice that 
	$$
	\sum_{k=0}^{mn_0-1} P_T(k,m\bn) = \sum_{d_0,d_1,\ldots,d_r \ge 0} e_{(d_0,\dd)}\left(I\mid J_1,\ldots,J_r\right)\binom{mn_0+d_0}{d_0+1}\binom{mn_1+d_1}{d_1}\cdots \binom{mn_r+d_r}{d_r}.
	$$
	Therefore, we obtain that $F(n_0,\bn)$ coincides with the following polynomial
	$$
	F(n_0,\bn) = \sum_{d_0+|\dd| = d-1} \frac{1}{(d_0+1)!\dd!}\, e_{(d_0,\dd)}\left(I\mid J_1,\ldots,J_r\right)\, n_0^{d_0+1}\bn^\dd,
	$$
	and so the result follows.
\end{proof}

Let $c$ be as in \autoref{lem_eq_filt}. For ease of notation, we define the following functions (recall \autoref{not_the_p})
\begin{equation}\label{Hp}
H_p(n_0,\bn) := \lim_{m\to \infty} \frac{\dim_{\kk}\left(\bJ(p)^{m\bn} / \left(\mm^{cmp(n_0+|\bn|)+1} \, \cap \, \bJ(p)^{m\bn}\right) \right)}{m^dp^d} 
\end{equation}
and 
\begin{equation}\label{Hphat}
\widehat{H}_p(n_0,\bn) := \lim_{m\to \infty} \frac{\dim_{\kk}\left(I_p^{mn_0}\bJ(p)^{m\bn} / \left(\mm^{cmp(n_0+|\bn|)+1} \, \cap \, I_p^{mn_0}\bJ(p)^{m\bn}\right) \right)}{m^dp^d}. 
\end{equation}
We note that the existence of these limits follow as in \autoref{eq_F_as_vols_general}.

The following technical proposition is needed to  treat the Noetherian case of the formula.

\begin{proposition}
	\label{prop_approx_Noeth_case}
	Assume \autoref{setup_mixed_mult_general}.
	In addition, suppose that $\II, \JJ(1),\ldots,\JJ(r)$ are Noetherian graded families.
	Then, for fixed $n_0 \in \NN$, $\bn \in \NN^r$ and $\varepsilon\in \RR_{>0}$,  there exists $p_0\in \NN$ such that if $p \gs p_0$ then 
	$$
	\Vol_{d}\left(\Delta(\Gamma_{n_0,\bn})\right) \,\ge \,
	H_p(n_0,\bn) \,\ge\, \Vol_{d}\left(\Delta(\Gamma_{n_0,\bn})\right)  - \varepsilon
	$$ 
	and
	$$ \Vol_{d}\big(\Delta(\widehat{\Gamma}_{n_0,\bn})\big) 
	\,\ge\, 
	\widehat{H}_p(n_0,\bn)
	\,\ge \,
	\Vol_{d}\big(\Delta(\widehat{\Gamma}_{n_0,\bn})\big) -\varepsilon.
	$$
\end{proposition}
\begin{proof}
	Similarly to \autoref{eq_def_Gamma} and \autoref{eq_def_Gamma_hat} we now define 
	\begin{equation*}
		\Gamma_{n_0,\bn}(p) :=\Big\{(\bm,mp) \in \NN^{d+1}\mid \bx^\bm \in \bJ(p)^{m \bn} \mbox{ and } |\bm| \le c mp (n_0+\lvert \bn \rvert)  \Big\}
	\end{equation*}
	and 
	\begin{equation*}
		\widehat{\Gamma}_{n_0,\bn}(p) :=\Big\{(\bm,mp)\in \NN^{d+1}\mid \bx^\bm \in I_{p}^{mn_0}\bJ(p)^{m \bn} \mbox{ and } |\bm| \le c m p (n_0+\lvert \bn \rvert)  \Big\}.
	\end{equation*}
	For each $(n_0,\bn) \in \NN^{d+1}$,  we consider the graded families of monomial ideals 
	$
	\{\bJ(p)^{\bn}\}_{p \in \NN} = \{J(1)_p^{n_1}\cdots J(r)_p^{n_r}\}_{p \in \NN}$ and $\{I_p^{n_0}\bJ(p)^{\bn}\}_{p\in \NN} = \{I_p^{n_0}J(1)_p^{n_1}\cdots J(r)_p^{n_r}\}_{p \in \NN}.
	$
	From these graded families we define the semigroups
	\begin{equation*}
		\mathfrak{A}_{n_0,\bn} :=\Big\{(\bm,p) \in \NN^{d+1}\mid \bx^\bm \in \bJ(p)^{\bn} \mbox{ and } |\bm| \le c p (n_0+\lvert \bn \rvert)  \Big\}
	\end{equation*}
	and 
	\begin{equation*}
		\mathfrak{B}_{n_0,\bn} :=\Big\{(\bm,p)\in \NN^{d+1}\mid \bx^\bm \in I_p^{n_0}\bJ(p)^{\bn} \mbox{ and } |\bm| \le c p (n_0+\lvert \bn \rvert)  \Big\}.
	\end{equation*}
	By construction, for all $p,m \ge 1$ we have the inclusions
	$$
	m \star \big[\mathfrak{A}_{n_0,\bn}\big]_{p} \subset \left[\Gamma_{n_0,\bn}(p) \right]_{mp} \subset  \big[\mathfrak{A}_{n_0,\bn}\big]_{mp}
	\;\; \text{ and } \;\; m \star \big[\mathfrak{B}_{n_0,\bn}\big]_{p} \subset \left[\widehat{\Gamma}_{n_0,\bn}(p) \right]_{mp} \subset  \big[\mathfrak{B}_{n_0,\bn}\big]_{mp}.
	$$
	As a consequence of \cite[Proposition 3.1]{lazarsfeld09} (also, see \cite[Theorem 3.3]{cutkosky2014}) and \autoref{thm_limit_KK}, for a fixed $\varepsilon\in \RR_{>0}$,  there exists $p_0\in \NN$ such that if $p \gs p_0$ then 
	\begin{equation*}
		\Vol_{d}\left(\Delta(\mathfrak{A}_{n_0,\bn})\right) 
		\ge  
		\lim_{m\to \infty} \frac{\#\big(\left[\Gamma_{n_0,\bn}(p) \right]_{mp} \big)}{m^dp^d} 
		\ge 
		\lim_{m\to \infty} \frac{\#\big(m \star \big[\mathfrak{A}_{n_0,\bn}\big]_{p}\big)}{m^dp^d} 
		\ge 
		\Vol_{d}\left(\Delta(\mathfrak{A}_{n_0,\bn})\right)    - \varepsilon
	\end{equation*}
	and
	\begin{equation*}
	\Vol_{d}\left(\Delta(\mathfrak{B}_{n_0,\bn})\right) 
	 \ge  
     \lim_{m\to \infty} \frac{\#\big(\big[\widehat{\Gamma}_{n_0,\bn}(p) \big]_{mp} \big)}{m^dp^d} 
	\ge 
	\lim_{m\to \infty} \frac{\#\big(m \star \big[\mathfrak{B}_{n_0,\bn}\big]_{p}\big)}{m^dp^d} 
	\ge 
	\Vol_{d}\left(\Delta(\mathfrak{B}_{n_0,\bn})\right)    - \varepsilon
	\end{equation*}
	Therefore, using the notation of \autoref{Hp} and \autoref{Hphat}, we obtain the inequalities  $		\Vol_{d}\left(\Delta(\mathfrak{A}_{n_0,\bn})\right) 
	\ge 
	 H_p(n_0,\bn)
	\ge 
	\Vol_{d}\left(\Delta(\mathfrak{A}_{n_0,\bn})\right)    - \varepsilon
	$
	and 
	$\Vol_{d}\left(\Delta(\mathfrak{B}_{n_0,\bn})\right) 
	\ge 
	\widehat{H}_p(n_0,\bn)
	\ge 
	\Vol_{d}\left(\Delta(\mathfrak{B}_{n_0,\bn})\right)    - \varepsilon
	$
	for all $p \ge p_0$.
	
	To conclude the proof we only need to show that the equalities $\Vol_{d}\left(\Delta(\mathfrak{A}_{n_0,\bn})\right)   = \Vol_{d}\left(\Delta(\Gamma_{n_0,\bn})\right)$ and $\Vol_{d}\left(\Delta(\mathfrak{B}_{n_0,\bn})\right)   = \Vol_{d}\left(\Delta(\widehat{\Gamma}_{n_0,\bn})\right)$ hold.

	By the Noetherian assumption, there exists $q > 0$ such that 
	\begin{equation}
		\label{eq_Noetherian}
			J(i)_{q}^n=J(i)_{nq} \;\; \text{ and } \;\; I_{q}^n=I_{nq} \;\; \text{ for every } n\gs 0, 1\ls i\ls r
	\end{equation}
	(see, e.g., \cite[Lemma 13.10]{GORTZ_WEDHORN}, \cite[Theorem 2.1]{herzog2007}). 
	Hence $\bJ(mq)^{\bn} = \bJ_{mq\bn}$ and $I_{mq}^{n_0}\bJ(mq)^{\bn} = I_{mqn_0}\bJ_{mq\bn}$ for all $m \ge 0$, and so \autoref{thm_limit_KK} yields the required equalities
	\begin{equation*}
	\Vol_{d}\left(\Delta(\Gamma_{n_0,\bn})\right) 
	=
	\lim_{m\to \infty} \frac{\#\big[\Gamma_{n_0,\bn}\big]_{mq}}{m^dq^d} 
	=
	\lim_{m\to \infty} \frac{\#\big[\mathfrak{A}_{n_0,\bn}\big]_{mq}}{m^dq^d} 
	=  
	\Vol_{d}\left(\Delta(\mathfrak{A}_{n_0,\bn})\right)
	\end{equation*}
	and
	\begin{equation*}
	\Vol_{d}\left(\Delta(\widehat{\Gamma}_{n_0,\bn})\right) 
	=
	\lim_{m\to \infty} \frac{\#\big[\widehat{\Gamma}_{n_0,\bn}\big]_{mq}}{m^dq^d} 
	=
	\lim_{m\to \infty} \frac{\#\big[\mathfrak{B}_{n_0,\bn}\big]_{mq}}{m^dq^d} 
	=  
	\Vol_{d}\left(\Delta(\mathfrak{B}_{n_0,\bn})\right). 
	\end{equation*}
	Therefore, the proof of the proposition is now complete.
\end{proof}

We now focus on approximating the graded families $\II, \JJ(1),\ldots,\JJ(r)$ by using successive truncations of them.
For that, we need to introduce some additional notation.

\begin{notation}
	\label{notation_truncations}
	Let $a>0$ be a positive integer. 
	Let $\II_a=\{I_{a, n}\}_{n\in \NN}$ be the Noetherian graded family generated by $I_1,\ldots, I_a$, that is, for $n>a$ one has $I_{a,n}=\sum_{i=1}^{n-1} I_{a,i}I_{a,n-i}$.   
	Likewise, define $\JJ(i)_a=\{J(i)_{a,n}\}_{n\in\NN}$ for all $1 \le i \le r$.	
\end{notation}

	For a vector $\bn=(n_1,\ldots, n_r) \in \NN^{r}$, we  abbreviate $\bJ_{a,\bn}=J(1)_{a,n_1}\cdots J(r)_{a,n_r}$. 
	As in \autoref{eq_def_Gamma} and \autoref{eq_def_Gamma_hat}, we now define 
	\begin{equation*}
	\Gamma_{a,n_0,\bn} :=\Big\{(\bm,m)=(m_1,\ldots, m_d,m)\in \NN^{d+1}\mid \bx^\bm \in \bJ_{a,m \bn} \mbox{ and } |\bm| \le c m (n_0+\lvert \bn \rvert)  \Big\}
	\end{equation*}
	and 
	\begin{equation*}
	\widehat{\Gamma}_{a,n_0,\bn} :=\Big\{(\bm,m)=(m_1,\ldots, m_d,m)\in \NN^{d+1}\mid \bx^\bm \in I_{a,mn_0}\bJ_{a,m \bn} \mbox{ and } |\bm| \le c m (n_0+\lvert \bn \rvert)  \Big\}.
	\end{equation*}
	For every $p \in \NN$ and $\bn = (n_1,\ldots,n_r) \in \NN^r$, let $\bJ(a,p)^\bn$ denote the ideal $J(1)_{a,p}^{n_1}\cdots J(r)_{a,p}^{n_r}$.
	Additionally, we have the following functions
	$$
	H_{a,p}(n_0,\bn) := \lim_{m\to \infty} \frac{\dim_{\kk}\left(\bJ(a,p)^{m\bn} / \left(\mm^{cmp(n_0+|\bn|)+1} \, \cap \, \bJ(a,p)^{m\bn}\right) \right)}{m^dp^d} 
	$$
	and 
	$$
	\widehat{H}_{a,p}(n_0,\bn) := \lim_{m\to \infty} \frac{\dim_{\kk}\left(I_{a,p}^{mn_0}\bJ(a,p)^{m\bn} / \left(\mm^{cmp(n_0+|\bn|)+1} \, \cap \, I_{a,p}^{mn_0}\bJ(a,p)^{m\bn}\right) \right)}{m^dp^d}. 
	$$
	
	The next technical proposition is used in the proof of \autoref{thm_mult_eq_vol}.
	For its proof we use an argument quite similar to the one used in \cite[Proposition 4.3]{cutkosky2019}.
	
\begin{proposition}
		\label{prop_approx_Vol}
		Assume \autoref{setup_mixed_mult_general}.
		Then, for fixed $n_0 \in \NN$, $\bn \in \NN^r$ and $\varepsilon\in \RR_{>0}$,  there exists $a_0\in \NN$ such that if $a\gs a_0$ then 
		$$
		\Vol_{d}\left(\Delta(\Gamma_{n_0,\bn})\right) 
		\,\ge\,  
		\Vol_{d}\left(\Delta(\Gamma_{a,n_0,\bn})\right) 
		\,\ge\, 
		\Vol_{d}\left(\Delta(\Gamma_{n_0,\bn})\right) - \varepsilon
		$$
		and 
		$$
		\Vol_{d}\big(\Delta(\widehat{\Gamma}_{n_0,\bn})\big) 
		\,\ge\,  
		\Vol_{d}\left(\Delta(\widehat{\Gamma}_{a,n_0,\bn})\right) 
		\,\ge\, 
		\Vol_{d}\big(\Delta(\widehat{\Gamma}_{n_0,\bn})\big) - \varepsilon.
		$$
\end{proposition}
\begin{proof}
		For a positive integer $a > 0$, since $I_n = I_{a,n}$ and $J(i) = J(i)_{a,n}$ for all $n \le a$, for all $m \ge 1$ we obtain the following inclusions
		$$
		m \star \left[\Gamma_{n_0,\bn}\right]_{a} \, \subset \, \left[\Gamma_{a,n_0,\bn} \right]_{ma} \, \subset \, \left[\Gamma_{n_0,\bn}\right]_{ma}
		\;\; \text{ and } \;\;
		m \star \big[\widehat{\Gamma}_{n_0,\bn}\big]_{a} \, \subset \, \big[\widehat{\Gamma}_{a,n_0,\bn} \big]_{ma} \, \subset \, \big[\widehat{\Gamma}_{n_0,\bn}\big]_{ma}.
		$$
		Then, by \cite[Proposition 3.1]{lazarsfeld09} (also, see \cite[Theorem 3.3]{cutkosky2014}) and \autoref{thm_limit_KK}, for a fixed $\varepsilon\in \RR_{>0}$,  there exists $a_0\in \NN$ such that if $a\gs a_0$ then 
		\begin{equation*}
			\Vol_{d}\left(\Delta(\Gamma_{n_0,\bn})\right) 
			\ge  
			\Vol_{d}\left(\Delta(\Gamma_{a,n_0,\bn})\right) 
			\ge 
			\lim_{m\to \infty} \frac{\#\big(m \star \left[\Gamma_{n_0,\bn}\right]_{a}\big)}{m^da^d} 
			\ge 
			\Vol_{d}\left(\Delta(\Gamma_{n_0,\bn})\right) - \varepsilon
		\end{equation*}
		and
		\begin{equation*}
			\Vol_{d}\big(\Delta(\widehat{\Gamma}_{n_0,\bn})\big) 
			\ge  
			\Vol_{d}\left(\Delta(\widehat{\Gamma}_{a,n_0,\bn})\right) 
			\ge 
			\lim_{m\to \infty} \frac{\#\big(m \star \big[\widehat{\Gamma}_{n_0,\bn}\big]_{a}\big)}{m^da^d} 
			\ge 
			\Vol_{d}\big(\Delta(\widehat{\Gamma}_{n_0,\bn})\big) - \varepsilon.
		\end{equation*}
	So, the result follows.
\end{proof}

Finally, we are ready for our analog of the ``Volume = Multiplicity formula'' in the case of mixed multiplicities. 
The following theorem expresses the mixed multiplicities of graded families as  normalized limits of  mixed multiplicities of ideals. 

\begin{theorem}
	\label{thm_mult_eq_vol}
	Assume \autoref{setup_mixed_mult_general}.
	Then, for each $d_0 \in \NN$ and $\dd = (d_1,\ldots,d_r) \in \NN^r$ with $d_0+|\dd| = d-1$, we have the equality 
	$$
	e_{(d_0,\dd)}\left(\II|\JJ(1),\ldots,\JJ(r)\right) \;=\; \lim_{p\to \infty} \frac{e_{(d_0,\dd)}\left(I_p|J(1)_p,\ldots,J(r)_p\right)}{p^d}.
	$$
\end{theorem}

\begin{proof}
	Let $p \ge 1$ and consider the filtrations $\II(p) = \{I_p^n\}_{n \in \NN}, \JJ(1)(p) = \{J(1)_p^n\}_{n \in \NN}, \ldots, \JJ(r)(p) = \{J(r)_p^n\}_{n \in \NN}$.
	By applying \autoref{thm_general_case} to the filtrations $$\II(p) = \{I_p^n\}_{n \in \NN},\, \JJ(1)(p) = \{J(1)_p^n\}_{n \in \NN},\, \ldots\,, \,\JJ(r)(p) = \{J(r)_p^n\}_{n \in \NN},$$ let $F_p(n_0,\bn)$ be the function 
	\begin{equation}
	\label{eq_def_F_p}
	F_p(n_0,\bn) = \lim_{m\to \infty} \frac{\dim_{\kk}\left(\bJ(p)^{m\bn}/I_p^{mn_0}\bJ(p)^{m\bn}\right) }{m^d}
	\end{equation}	
	and $G_p(n_0,\bn)$ be the polynomial of total degree $d$ that coincides with $F_p(n_0,\bn)$.
	From \autoref{lem_eq_mixed_mult_two_defs} we can write 
	\begin{equation}
	\label{eq_G_p}
	G_p(n_0,\bn) = \sum_{d_0+|\dd| = d-1} \frac{1}{(d_0+1)!\dd!}\, e_{(d_0,\dd)}\left(I_p|J(1)_p,\ldots,J(r)_p\right)\, n_0^{d_0+1}\bn^\dd.
	\end{equation}	
	Notice that we have the equation
	\begin{equation}
		\label{eq_F_expressed_as_H}
		\frac{1}{p^d}F_p(n_0,\bn) = H_p(n_0,\bn)  - \widehat{H}_p(n_0,\bn); 
	\end{equation}
see \autoref{Hp} and \autoref{Hphat}. 

Fix $\varepsilon > 0$  a positive real number.
	By using \autoref{prop_approx_Vol}, choose $a \gg 0$ such that 
	\begin{equation}
		\label{eq_ineq_ee_2_Vol_1}
		\Vol_{d}\left(\Delta(\Gamma_{n_0,\bn})\right) 
		\,\ge\,  
		\Vol_{d}\left(\Delta(\Gamma_{a,n_0,\bn})\right) 
		\,\ge\, 
		\Vol_{d}\left(\Delta(\Gamma_{n_0,\bn})\right) - \varepsilon/2
	\end{equation}
	and 
	\begin{equation}
		\label{eq_ineq_ee_2_Vol_2}
		\Vol_{d}\big(\Delta(\widehat{\Gamma}_{n_0,\bn})\big) 
		\,\ge\,  
		\Vol_{d}\left(\Delta(\widehat{\Gamma}_{a,n_0,\bn})\right) 
		\,\ge\, 
		\Vol_{d}\big(\Delta(\widehat{\Gamma}_{a,n_0,\bn})\big) - \varepsilon/2.
	\end{equation}
	From \autoref{prop_approx_Noeth_case}, applied to the Noetherian graded families $\II_a, \JJ(1)_a,\ldots,\JJ(r)_a$, choose $p \gg 0$ such that 
	\begin{equation}
		\label{eq_ineq_ee_2_Vol_3}
		\Vol_{d}\left(\Delta(\Gamma_{a,n_0,\bn})\right) 
		\,\ge\,  
		H_{a,p}(n_0,\bn) \,\ge\, \Vol_{d}\left(\Delta(\Gamma_{a,n_0,\bn})\right)  - \varepsilon/2 
	\end{equation}
 	and 	
 	\begin{equation}
 		\label{eq_ineq_ee_2_Vol_4}
 		\Vol_{d}\big(\Delta(\widehat{\Gamma}_{a,n_0,\bn})\big) 
 		\,\ge\,  
 		\widehat{H}_{a,p}(n_0,\bn) \,\ge\, \Vol_{d}\big(\Delta(\widehat{\Gamma}_{a,n_0,\bn})\big) - \varepsilon/2.
 	\end{equation}
	Since $I_{a,n} \subset I_n$ and $J(i)_{a,n} \subset J(i)_n$ for all $n \in \NN$,  one has $H_p(n_0,\bn) \ge H_{a,p}(n_0,\bn)$ and $\widehat{H}_p(n_0,\bn) \ge \widehat{H}_{a,p}(n_0,\bn)$, and so from \autoref{eq_ineq_ee_2_Vol_1}, \autoref{eq_ineq_ee_2_Vol_2}, \autoref{eq_ineq_ee_2_Vol_3} and \autoref{eq_ineq_ee_2_Vol_4} one obtains 
	\begin{equation}
		\label{eq_ineq_Vol_H1}
		\Vol_{d}\left(\Delta(\Gamma_{n_0,\bn})\right) 
		\,\ge\,  
		H_p(n_0,\bn) \,\ge\, H_{a,p}(n_0,\bn)
		\,\ge\, 
		\Vol_{d}\left(\Delta(\Gamma_{n_0,\bn})\right) - \varepsilon
	\end{equation}
	and 
	\begin{equation}
		\label{eq_ineq_Vol_H2}
		\Vol_{d}\big(\Delta(\widehat{\Gamma}_{n_0,\bn})\big) 
		\,\ge\,  
		\widehat{H}_p(n_0,\bn) \,\ge\, \widehat{H}_{a,p}(n_0,\bn) 
		\,\ge\, 
		\Vol_{d}\big(\Delta(\widehat{\Gamma}_{n_0,\bn})\big) - \varepsilon.
	\end{equation}

	Therefore, by combining \autoref{eq_F_as_vols_general}, \autoref{eq_F_expressed_as_H}, \autoref{eq_ineq_Vol_H1} and \autoref{eq_ineq_Vol_H2} we obtain the equalities
	\begin{align*}
		F(n_0,\bn) &= \Vol_{d}\left(\Delta(\Gamma_{n_0,\bn})\right) - \Vol_{d}\left(\Delta(\widehat{\Gamma}_{n_0,\bn})\right) \\ 
		&= \lim_{p\to \infty} \left(H_p(n_0,\bn)  - \widehat{H}_p(n_0,\bn)
		\right)	
		=
		\lim_{p\to \infty} \frac{1}{p^d} F_p(n_0,\bn)
	\end{align*}
	for all $n_0$ and $\bn \in \NN^r$.
	Accordingly, it necessarily follows that the coefficients of the polynomials $\frac{1}{p^d}G_p(n_0,\bn)$ converge to the ones of the polynomial $G(n_0,\bn)$ (see, e.g., \cite[Lemma 3.2]{cutkosky2019}).
	Therefore, by \autoref{def_mix_mult_general} and \autoref{eq_G_p} we obtain  
	$$
	e_{(d_0,\dd)}\left(\II|\JJ(1),\ldots,\JJ(r)\right) \;=\; \lim_{p\to \infty} \frac{e_{(d_0,\dd)}\left(I_p|J(1)_p,\ldots,J(r)_p\right)}{p^d},
	$$
	and so the result follows.
\end{proof}

\section{Mixed volumes of Convex Bodies as Mixed Multiplicities}
\label{sec_mixed_vol_mixed_mult}

This section includes the proof  \autoref{thmC}, which is the main result of this paper (see \autoref{mainCor}). In this result, we describe the mixed volumes of (arbitrary) convex bodies as the mixed multiplicities of certain (not necessarily Noetherian) graded families of ideals,  and as the normalized limits of the mixed multiplicities of  certain monomial ideals.

Throughout this section we fix the following setup.

\begin{setup}
	\label{setup_mixed_vol}
	Let $\kk$ be a field, $R$ be the standard graded polynomial ring $R = \kk[x_1,\ldots,x_{d+1}]$, and $\mm \subset R$ be the graded irrelevant ideal $\mm=\left(x_1,\ldots,x_{d+1}\right)$.
	Let 
	$
	\pi_1 : \RR^{d+1} \rightarrow \RR^{d} , \, (\alpha_1,\ldots,\alpha_d,\alpha_{d+1}) \mapsto (\alpha_1,\ldots,\alpha_d)
	$ be the projection into the first $d$ factors. 
	Let $\pi : \RR^{d+1} \rightarrow \RR$ the linear map 
	$
	\pi : \RR^{d+1} \rightarrow \RR, \, (\alpha_1,\ldots,\alpha_d,\alpha_{d+1}) \mapsto \alpha_1+\cdots+\alpha_d+\alpha_{d+1}.
	$	
	Let $(K_1,\ldots,K_r)$ be a sequence of convex bodies in $\RR_{\ge 0}^d$.
\end{setup}

The notation below introduces a process that we call the homogenization of a convex body.

\begin{notation}
	Let $K \subset \RR_{\ge 0}^d$ be a convex body in $\RR_{\ge 0}^d$.
	Choose $h_K \in \NN$ a positive integer such that $K \subset \pi_1\big(\pi^{-1}(h_K) \cap \RR_{\ge 0}^{d+1}\big)$.
	We call $h_K$ a \emph{suitable degree of homogenization of $K$}.
	The corresponding \emph{homogenization} of $K$ (with respect to $h_{K}$) is defined as the convex body 
	$$
	\widetilde{K} \;:=\; \left(K \times \RR\right) \cap \pi^{-1}(h_K) \; \subset \; \RR_{\ge 0}^{d+1}.
	$$
	Let $C_K$ be the corresponding cone $C_K := \text{Cone}(\widetilde{K})$.
	Consider the semigroup $S_K \subset \NN^{d+1}$ given by
	$$
	S_K:= C_K \cap \NN^{d+1} \cap \left(\bigcup_{k=1}^\infty \pi^{-1}\left(kh_K\right)\right).
	$$
\end{notation}

For each $1 \le i \le r$, let $h_{K_i}$ be a suitable degree of homogenization for $K_i$ and $S_{K_i}$ be the corresponding semigroup in $\NN^{d+1}$ that is determined by the homogenization $\widetilde{K_i} \subset \RR_{\ge 0}^{d+1}$.
	For each $1 \le i \le r$, we consider the (not necessarily Noetherian) graded family of monomial ideals
\begin{equation}
	\label{eq_def_grad_families}
		\JJ(i) = \{J(i)_n\}_{n \in \NN},  \quad\text{where }\quad J(i)_n = \Big(\bx^\bm \mid \bm \in S_{K_i} \text{ and } \lvert \bm \rvert = nh_{K_i}\Big) \subset R.  
\end{equation}
	Let $K_0 \subset \RR^{d}$ be the convex hull of the points $\mathbf{0}, \ee_1,\ldots,\ee_{d} \in \RR^{d}$ and $\widetilde{K_0}$ be its homogenization $\widetilde{K_0} = \lbrace \bx \in \RR_{\ge 0}^{d+1} \mid \lvert \bx \rvert = 1 \rbrace \subset \RR^{d+1}$; notice that $h_{K_0}=1$ is a suitable degree of homogenization for $K_0$. We let $\mathbb{M}$ be the graded family $\mathbb{M}=\{\mm^n\}_{n\in \NN}$.
 Denote by $\bK$ the sequence of convex bodies $\bK = (K_0,K_1,\ldots,K_r)$.

Let $p > 0$ be a positive integer. 
For each $1 \le i \le r$, let $K_i(p)$ be the lattice polytope given by
$$
K_i(p) \;:=\; \pi_1\left( \text{conv}\left( \Big\lbrace 
\bm  \in \NN^{d+1} \mid \bx^\bm \in  \left[J(i)_{p}\right]_{ph_{K_i}} 
\Big\rbrace\right) \right),
$$
$\text{conv}(-)$ denotes the convex hull of a finite set of points; the polytope defined above corresponds with the generators of the ideal $J(i)_p$.
Denote by $\bK(p)$ the sequence of lattice polytopes $\bK(p) := (K_0, K_1(p), \ldots, K_r(p))$.
The next lemma shows that the mixed volumes of $\bK$ can be approximated with the ones of $\bK(p)$.

\begin{lemma}
	\label{lemma_aprox_MV}
	For each $(d_0,\dd) \in \NN^{r+1}$ with $d_0+\lvert \dd \rvert = d$, we have the equality 
	$$
	\MV_d\left(\bK_{(d_0,\dd)}\right)  \;=\; \lim_{p\to \infty} \frac{\MV_d\left(\bK(p)_{(d_0,\dd)}\right)}{p^{|\dd|}}.
	$$
\end{lemma}
\begin{proof}
		By construction we have that $\frac{1}{p}K_i(p)$ converges to $K_i$ in the Hausdorff metric (see \cite[Definition 2.1, page 109]{EWALD}) when $p \rightarrow \infty$.
		  Thus, \cite[Lemma 3.8, page 119]{EWALD} yields that 
		  $$
		  \MV_d\left({\left(K_0, \frac{1}{p}K_1(p),\ldots,\frac{1}{p}K_r(p)\right)}_{(d_0,\dd)}\right)
		  $$ converges to $\MV_d\left(\bK_{(d_0,\dd)}\right)$ when $p \rightarrow \infty$.
		  Therefore, the result follows from the linearity of mixed volumes (see, e.g., \cite[Lemma 3.6, page 118]{EWALD}).
\end{proof}

Finally, the next theorem expresses the mixed volume of the convex bodies $K_1,\ldots,K_r$ as a mixed multiplicity of the chosen graded families $\JJ(1), \ldots, \JJ(r)$.
Additionally, we also express the mixed volume of the convex bodies $K_1,\ldots,K_r$ as two types of normalized limits of mixed multiplicities of ideals.
This result can be seen as an extension of \cite[Theorem 2.4, Corollary 2.5]{TRUNG_VERMA_MIXED_VOL}. 

\begin{theorem}\label{mainCor}
	Assume \autoref{setup_mixed_vol}.
	Let $\JJ(1),\ldots,\JJ(r)$ be the graded families of monomial ideals defined in \autoref{eq_def_grad_families}.
	For each $(d_0,\dd) \in \NN^{r+1}$ with $d_0+\lvert \dd \rvert = d$, we have the equalities 
	\begin{align*}
		\MV_d\left(\bK_{ (d_0,\dd)}\right) &=  e_{(d_0,\dd)}(\mathbb{M} \mid \JJ(1),\ldots,\JJ(r)) \\
		&= \lim_{p\to \infty} \frac{e_{(d_0,\dd)}(\mm^p \mid J(1)_p,\ldots,J(r)_p)}{p^{d+1}} \;=\; \lim_{p\to \infty} \frac{e_{(d_0,\dd)}(\mm \mid J(1)_p,\ldots,J(r)_p)}{p^{|\dd|}}.
	\end{align*}
	In particular, when $r = d$, we obtain the equalities
	\begin{align*}
		\MV_d(K_1,\ldots,K_d) &=  e_{(0,1,\ldots,1)}(\mathbb{M} \mid \JJ(1),\ldots,\JJ(d)) \\
		&= \lim_{p\to \infty} \frac{e_{(0,1,\ldots,1)}(\mm^p \mid J(1)_p,\ldots,J(d)_p)}{p^{d+1}} \;=\; \lim_{p\to \infty} \frac{e_{(0,1,\ldots,1)}(\mm \mid J(1)_p,\ldots,J(d)_p)}{p^{d}}.
	\end{align*}
\end{theorem}
\begin{proof}
	By applying \cite[Theorem 2.4, Corollary 2.5]{TRUNG_VERMA_MIXED_VOL} to the generators of the monomial ideals $J(1)_p,\ldots,J(r)_p$, we obtain the equality $\MV_d\left(\bK(p)_{(d_0,\dd)}\right)=e_{(d_0,\dd)}(\mm \mid J(1)_p,\ldots,J(r)_p)$.
	Thus, \autoref{lemma_aprox_MV} yields the equality 
	$$
	\MV_d\left(\bK_{ (d_0,\dd)}\right) = \lim_{p\to \infty} \frac{e_{(d_0,\dd)}(\mm \mid J(1)_p,\ldots,J(r)_p)}{p^{|\dd|}}.
	$$
	
	From \autoref{thm_general_case}, let $G_1(n_0,\bn)$ and $G_2(n_0,\bn)$ be the polynomials given by the functions
	$$
	G_1(n_0,\bn) = \lim_{m\rightarrow \infty}\frac{\dim_\kk \left( J(1)_p^{mn_1}\cdots J(r)_p^{mn_r}/\mm^{mn_0}J(1)_p^{mn_1}\cdots J(r)_p^{mn_r} \right)}{m^{d+1}}
	$$
	and 
	$$
	G_2(n_0,\bn) = \lim_{m\rightarrow \infty}\frac{\dim_\kk \left( J(1)_p^{mn_1}\cdots J(r)_p^{mn_r}/{\left(\mm^p\right)}^{mn_0}J(1)_p^{mn_1}\cdots J(r)_p^{mn_r} \right)}{m^{d+1}},
	$$
	respectively.
	Due to \autoref{lem_eq_mixed_mult_two_defs} we have that 
	$$
	G_1(n_0,\bn) = \sum_{d_0+|\dd| = d} \frac{1}{(d_0+1)!\dd!}\, e_{(d_0,\dd)}\left(\mm|J(1)_p,\ldots,J(r)_p\right)\, n_0^{d_0+1}\bn^\dd.
	$$
	and 
	$$
	G_2(n_0,\bn) = \sum_{d_0+|\dd| = d} \frac{1}{(d_0+1)!\dd!}\, e_{(d_0,\dd)}\left(\mm^p|J(1)_p,\ldots,J(r)_p\right)\, n_0^{d_0+1}\bn^\dd.
	$$
	Since $G_1(pn_0,\bn) = G_2(n_0,\bn)$, by comparing the coefficients of $G_1(pn_0,\bn)$ and $G_2(n_0,\bn)$, we obtain that $e_{(d_0,\dd)}\left(\mm^p|J(1)_p,\ldots,J(r)_p\right) = p^{d_0+1}e_{(d_0,\dd)}\left(\mm|J(1)_p,\ldots,J(r)_p\right)$.
	Hence the equality 
	$$ \lim_{p\to \infty} \frac{e_{(d_0,\dd)}(\mm^p \mid J(1)_p,\ldots,J(r)_p)}{p^{d+1}} \;=\; \lim_{p\to \infty} \frac{e_{(d_0,\dd)}(\mm \mid J(1)_p,\ldots,J(r)_p)}{p^{|\dd|}}
	$$
	 is clear.
	 Finally, the proof is concluded by invoking \autoref{thm_mult_eq_vol}.
\end{proof}

\section*{Acknowledgments}

The second  author is  supported by NSF Grant DMS \#2001645.

\bibliographystyle{elsarticle-num} 

\begin{bibdiv}
\begin{biblist}

\bib{bernshtein}{article}{
      author={Bernshtein, David~N},
       title={The number of roots of a system of equations},
        date={1975},
     journal={Functional Analysis and its applications},
      volume={9},
      number={3},
       pages={183\ndash 185},
}

\bib{MIXED_MULT}{article}{
      author={Cid-Ruiz, Yairon},
       title={Mixed multiplicities and projective degrees of rational maps},
        date={2021},
        ISSN={0021-8693},
     journal={Journal of Algebra},
      volume={566},
       pages={136 \ndash  162},
  url={http://www.sciencedirect.com/science/article/pii/S0021869320304671},
}

\bib{cutkosky2013}{article}{
      author={Cutkosky, Steven},
       title={Multiplicities associated to graded families of ideals},
        date={2013},
     journal={Algebra \& Number Theory},
      volume={7},
      number={9},
       pages={2059\ndash 2083},
}

\bib{cutkosky2019}{article}{
      author={Cutkosky, Steven},
      author={Sarkar, Parangama},
      author={Srinivasan, Hema},
       title={Mixed multiplicities of filtrations},
        date={2019},
     journal={Transactions of the American Mathematical Society},
      volume={372},
      number={9},
       pages={6183\ndash 6211},
}

\bib{cutkosky2014}{article}{
      author={Cutkosky, Steven~Dale},
       title={Asymptotic multiplicities of graded families of ideals and linear
  series},
        date={2014},
     journal={Advances in Mathematics},
      volume={264},
       pages={55\ndash 113},
}

\bib{cutkosky2015}{article}{
      author={Cutkosky, Steven~Dale},
       title={A general volume= multiplicity formula},
        date={2015},
     journal={Acta Mathematica Vietnamica},
      volume={40},
      number={1},
       pages={139\ndash 147},
}

\bib{ein2003}{article}{
      author={Ein, Lawrence},
      author={Lazarsfeld, Robert},
      author={Smith, Karen~E},
       title={Uniform approximation of abhyankar valuation ideals in smooth
  function fields},
        date={2003},
     journal={American Journal of Mathematics},
      volume={125},
      number={2},
       pages={409\ndash 440},
}

\bib{EWALD}{book}{
      author={Ewald, G\"{u}nter},
       title={Combinatorial convexity and algebraic geometry},
      series={Graduate Texts in Mathematics},
   publisher={Springer-Verlag, New York},
        date={1996},
      volume={168},
}

\bib{GORTZ_WEDHORN}{book}{
      author={G\"{o}rtz, Ulrich},
      author={Wedhorn, Torsten},
       title={Algebraic geometry {I}},
      series={Advanced Lectures in Mathematics},
   publisher={Vieweg + Teubner, Wiesbaden},
        date={2010},
        ISBN={978-3-8348-0676-5},
         url={https://doi.org/10.1007/978-3-8348-9722-0},
        note={Schemes with examples and exercises},
}

\bib{HERMANN_MULTIGRAD}{article}{
      author={Herrmann, Manfred},
      author={Hyry, Eero},
      author={Ribbe, J\"{u}rgen},
      author={Tang, Zhongming},
       title={Reduction numbers and multiplicities of multigraded structures},
        date={1997},
     journal={J. Algebra},
      volume={197},
      number={2},
       pages={311\ndash 341},
}

\bib{herzog2007}{article}{
      author={Herzog, J{\"u}rgen},
      author={Hibi, Takayuki},
      author={Trung, Ng{\^o}~Vi{\^e}t},
       title={Symbolic powers of monomial ideals and vertex cover algebras},
        date={2007},
     journal={Advances in Mathematics},
      volume={210},
      number={1},
       pages={304\ndash 322},
}

\bib{Huh12}{article}{
      author={Huh, June},
       title={Milnor numbers of projective hypersurfaces and the chromatic
  polynomial of graphs},
        date={2012},
     journal={Journal of the American Mathematical Society},
      volume={25},
      number={3},
       pages={907\ndash 927},
}

\bib{huneke2006integral}{book}{
      author={Huneke, Craig},
      author={Swanson, Irena},
       title={Integral closure of ideals, rings, and modules},
   publisher={Cambridge University Press},
        date={2006},
      volume={13},
}

\bib{KAVEH_KHOVANSKII}{article}{
      author={Kaveh, Kiumars},
      author={Khovanskii, Askold~G},
       title={Newton-{O}kounkov bodies, semigroups of integral points, graded
  algebras and intersection theory},
        date={2012},
     journal={Annals of Mathematics},
       pages={925\ndash 978},
}

\bib{kaveh2014}{article}{
      author={Kaveh, Kiumars},
      author={Khovanskii, Askold~G},
       title={Convex bodies and multiplicities of ideals},
        date={2014},
     journal={Proceedings of the Steklov Institute of Mathematics},
      volume={286},
      number={1},
       pages={268\ndash 284},
}

\bib{kho78}{article}{
      author={Khovanskii, Askold~G},
       title={Newton polyhedra and the genus of complete intersections},
        date={1978},
     journal={Functional Analysis and its applications},
      volume={12},
      number={1},
       pages={38\ndash 46},
}

\bib{kouchnirenko}{article}{
      author={Kouchnirenko, Anatoli~G},
       title={Polyedres de newton et nombres de milnor},
        date={1976},
     journal={Inventiones mathematicae},
      volume={32},
      number={1},
       pages={1\ndash 31},
}

\bib{lazarsfeld09}{inproceedings}{
      author={Lazarsfeld, Robert},
      author={Mustaț{\u{a}}, Mircea},
       title={Convex bodies associated to linear series},
        date={2009},
   booktitle={Annales scientifiques de l'{\'e}cole normale sup{\'e}rieure},
      volume={42},
       pages={783\ndash 835},
}

\bib{must02}{article}{
      author={Musta{\c{t}}ǎ, Mircea},
       title={On multiplicities of graded sequences of ideals},
        date={2002},
     journal={Journal of Algebra},
      volume={256},
      number={1},
       pages={229\ndash 249},
}

\bib{SCHNEIDER}{book}{
      author={Schneider, Rolf},
       title={Convex bodies: the {B}runn-{M}inkowski theory},
     edition={expanded},
      series={Encyclopedia of Mathematics and its Applications},
   publisher={Cambridge University Press, Cambridge},
        date={2014},
      volume={151},
}

\bib{TRUNG_VERMA_SURVEY}{article}{
      author={Trung, N.~V.},
      author={Verma, J.~K.},
       title={Hilbert functions of multigraded algebras, mixed multiplicities
  of ideals and their applications},
        date={2010},
     journal={J. Commut. Algebra},
      volume={2},
      number={4},
       pages={515\ndash 565},
}

\bib{TRUNG_VERMA_MIXED_VOL}{article}{
      author={Trung, Ngo~Viet},
      author={Verma, Jugal},
       title={Mixed multiplicities of ideals versus mixed volumes of
  polytopes},
        date={2007},
     journal={Trans. Amer. Math. Soc.},
      volume={359},
      number={10},
       pages={4711\ndash 4727},
}

\end{biblist}
\end{bibdiv}

\end{document}